\documentclass[11pt]{amsart}
\usepackage{xcolor}
\usepackage{float}
\usepackage[T1]{fontenc}
\usepackage[utf8]{inputenc}
\usepackage{lmodern}
\usepackage{amsmath,amssymb,amsthm,mathtools}
\usepackage{multirow}
\usepackage{booktabs}
\usepackage{microtype}
\usepackage[a4paper,margin=1in]{geometry}
\usepackage[hidelinks]{hyperref}
\usepackage{silence}
\usepackage{mathrsfs}
\newtheorem{theorem}{Theorem}[section]
\newtheorem{lemma}[theorem]{Lemma}
\newtheorem {proposition}[theorem]{Proposition}
\newtheorem {definition}[theorem]{Definition}
\newtheorem {example}[theorem]{Example}
\newtheorem {question}[theorem]{Question}

\newtheorem{remark}[theorem]{Remark}

\newcommand{\A}{\mathbb{A}}

\newcommand{\SO}{\operatorname{SO}}

\newtheorem {corollary} [theorem] {Corollary}

\newtheorem{mainthm}{Theorem}

\title{Profinite Non-Rigidity of Arithmetic Lattices and the K\"ahler Property}

\date{}
\author{Yukun Du}
\address{Department of Mathematics, University of Georgia, Athens, GA 30603}
 \email{yukun.du@uga.edu }
\urladdr{https://benniedu.github.io/}

\author{Feng Hao}
\address{School of Mathematics, Shandong University,
Jinan, 250100, P. R. China.}
\email{feng.hao@sdu.edu.cn}

\author{Kejia Zhu}
\address{School of Mathematics, Hunan University, Changsha, Hunan, 410082, P.R.China}
 \email{kzhumath@gmail.com}
\urladdr{https://sites.google.com/view/kejiazhu} 

\begin{document}

\begin{abstract}
We prove that the K\"ahlerness of a finitely presented, residually finite group is not determined by its profinite completion, which answers a question asked independently by Arapura and Libgober. In the proof, we give explicit constructions to characterize the absolute Dynkin types that admit non-isomorphic real forms of higher-rank Lie groups containing torsion-free arithmetic lattices with isomorphic profinite completions. 
\end{abstract}

\maketitle

\section{Introduction}
In 1959, Grothendieck \cite{grothendieck5geometrie} first introduced the algebraic fundamental group to study the fundamental group of varieties and schemes. For a complex algebraic variety $X$, its algebraic fundamental group $\pi_1^{alg}(X)$ is indeed the profinite completion of $\pi_1(X)$, i.e. $$\pi_1^{alg}(X)\cong \widehat{\pi_1(X)}:=\varprojlim_{\substack{N\trianglelefteq_f \pi_1(X)}} \pi_1(X)/N.$$  
Although the profinite completion of a group records all of its finite quotients, the passage from a discrete group to its profinite completion generally loses information. In the geometric setting, it is natural to ask to what extent $\pi_1(X)$ can be reconstructed from
$\widehat{\pi_1(X)}$. 

In 1970, Grothendieck \cite{grothendieck1970representations}
posed the following famous question: Let $f:G_1\to G_2$ be a homomorphism between finitely presented, residually finite groups. If the induced continuous homomorphism $\widehat f:\widehat{G}_1\to \widehat{G}_2$ is an isomorphism, then must $f$ itself be an isomorphism?

Grothendieck's question was answered in the negative by Bridson and Grunewald \cite{bridson2004grothendieck}. 
Moreover, Serre \cite{serre1964exemples} constructed conjugate smooth complex projective varieties $X$ and $X^\sigma$ such that $\widehat{\pi_1(X)}\cong \widehat{\pi_1(X^\sigma)}$ but ${\pi_1(X)}\not\cong {\pi_1(X^\sigma)}$. 


In Serre's example, $X$ is a smooth projective variety over $\mathbb C$, so $\pi_1(X)$ is a K\"ahler group, namely the fundamental group of a compact K\"ahler manifold. Thus, Serre's example shows that the profinite completion does not determine a K\"ahler group up to isomorphism, even among K\"ahler groups. Nevertheless, it leaves open a more geometric recognition problem. Indeed, the K\"ahler property imposes strong restrictions not only on a group itself, but also on its finite-index subgroups. Since the profinite completion records the system of finite-index subgroups of a group, it leads to the following fundamental question asked independently by Arapura and Libgober:

\begin{question}\label{question}
Is the K\"ahlerness of a residually finite group determined by its profinite completion?
\end{question}


In the main theorem below, we answer the Question \ref{question} negatively (see Theorem \ref{mainThem2} for the specific construction):

\begin{mainthm}\label{thm:main}
There exist infinitely many pairs of finitely presented, residually finite groups $(G_1,G_2)$ such that
$$
\widehat{G}_1\cong\widehat{G}_2,
$$
where $G_1$ is a K\"ahler group and $G_2$ is not a K\"ahler group.
\end{mainthm}

\begin{remark}
The restriction to residually finite K\"ahler groups is indispensable. Without residual finiteness, one immediately encounters pathological examples. Indeed, for any K\"ahler group $G$ and any infinite simple
group $S$, a classical example due to Gromov \cite{Gromov1989SurLG} gives $\widehat{G}\cong\widehat{G*S}$, while $G*S$ is not K\"ahler. The assumption that $G'$ is finitely presented is also natural from
a geometric perspective, since every K\"ahler group is finitely presented.\\
Although the main theorem shows that, in general, the K\"ahlerness of a finitely presented, residually finite group is not determined by its profinite completion, in \cite{hughes2025profinite}, Hughes, Llosa-Isenrich, Py, Stover, and Vidussi showed that 
certain groups are determined within the
class of residually finite K\"ahler groups by their profinite completion, including products of surface groups and certain groups with exotic finiteness.
\end{remark}

The proof of Theorem~\ref{thm:main} uses arithmetic lattices with isomorphic profinite completions in different higher-rank semisimple Lie groups. The existence of these lattices is established in Theorem~\ref{thm:ss_classical}(1), and an explicit construction is given in Theorem~\ref{prop:lattice_level_3}. The proof of Theorem~\ref{thm:ss_classical} relies on two main ingredients: Lemma~\ref{lem:CSP}, motivated by the work of Kammeyer and Kionke (see \cite{kammeyer2021adelic,kammeyer2023profinite}), and the construction of algebraic groups that are isomorphic at every finite place but differ in a prescribed way at certain archimedean places. Allowing multiple noncompact archimedean places enables us to include $\operatorname{Spin}(n,1)$-factors while maintaining the higher-rank condition. Theorem~\ref{prop:G2-not-kahler} then provides the obstruction to K\"ahlerness needed to distinguish the two lattices.

By \cite{kammeyer2021adelic},\cite{kammeyer2023profinite}, the ambient real form is an invariant of the abstract lattice, whereas profinitely isomorphic arithmetic lattices have the same absolute Dynkin type. This raises the question of which non-isomorphic real forms of a fixed absolute Dynkin type admit arithmetic lattices with isomorphic profinite completions. We answer this question in Theorem~\ref{Mainthem}, which summarizes Propositions~\ref{prop:classic}, \ref{prop:exceptional}, and \ref{prop:slh}, Theorem~\ref{thm:rigidity}, and Remark~\ref{remk:rigidity}.

\begin{mainthm}\label{Mainthem}
Let $G,G'$ be non-isomorphic higher-rank simply connected real forms of the same absolute type other than $E_6$. They admit torsion-free cocompact arithmetic lattices with isomorphic profinite completions if and only if they are a pair of compatible indefinite spin, unitary, or quaternionic symplectic groups of type $B_\ell,D_\ell$ ($\ell\ge4$), $A_\ell$ ($\ell\ge6$), or $C_\ell$ ($\ell\ge6$), real and quaternionic special linear groups of type $A_{4n-1}$ ($n\ge2$), or the split and minimally non-compact real forms of type $E_7$ or $E_8$.
\end{mainthm}

\begin{remark}
     After our first draft, Py and Isenrich pointed us to Aka's work. In \cite[Theorem 1]{aka2012profinite}, Aka also constructed pairs of simply-connected absolutely almost simple algebraic groups of the same classical absolute Dynkin type ($D_{(n+1)/2}$ if $n$ is odd and $B_{n/2}$ if $n$ is even) in $\mathrm{Spin}(1,n)^2$ and $\mathrm{Spin}(5,n-4)^2$ for $n\ge6$. Also, after our first draft, Holger Kammeyer and Adrian Baumann brought their recent work \cite{baumann2026congruence} to our attention. They establish classification results for congruence rigidity of algebraic groups over number fields. Our Theorem \ref{Mainthem} addresses the related question of which absolute Dynkin types admit distinct higher-rank real forms containing arithmetic lattices with isomorphic profinite completions. In Section \ref{section3}, we give explicit constructions of the arithmetic lattices for each case of Theorem \ref{Mainthem}. These constructions play an essential role in the proof of Theorem \ref{thm:main}.
\end{remark}

Theorem~\ref{Mainthem}, together with  Theorem~\ref{prop:lattice_level_3} (see Theorem~\ref{thm:ss_classical} and Proposition~\ref{prop:exceptional} for the explicit constructions),
can be used to study the profinite invariance of properties related to the ambient real form or the geometry of its symmetric space. We hope it can be useful in the forthcoming work: for example, note that the construction in this paper does not arise from the Grothendieck pair. It is natural to ask if the Grothendieck pair would help determine the K\"ahlerness:
\begin{question}\label{Ques:Grothendieck pair}
   Let $f:G_1\to G_2$ be a homomorphism between finitely presented, residually finite groups. If the induced continuous homomorphism $\widehat f:\widehat{G}_1\to \widehat{G}_2$ is an isomorphism, if $G_1$ (resp. $G_2$) is K\"ahler, then must $G_2$ (resp. $G_1$) be K\"ahler too? 
\end{question}

\noindent \textbf{Acknowledgment} Yukun Du would like to remember with gratitude his late Ph.D. advisor, Michael Kapovich. He is grateful for their many discussions on Lie groups and arithmetic lattices and dedicates his contribution to Kapovich's memory. Feng Hao would like to thank Donu Arapura for asking him Question \ref{question} and helpful discussion. He would also like to thank Claudio Llosa Isenrich for useful conversations. Kejia Zhu would like to thank Anatoly Libgober for asking him the same question and for useful conversations. He would also like to thank Shaver Phagan, Pierre Py, Corey Bregman, and Daniel Groves for useful discussions. We also thank Pierre Py and Claudio Llosa Isenrich for telling us about the work of Aka (\cite{aka2012profinite}) after our first draft, which we did not notice. We also thank Holger Kammeyer and Adrian Baumann for showing us their recent work \cite{baumann2026congruence} after our first draft, which we did not notice.

Feng Hao is supported by  NSFC No. 1240010723, SDNSFC (No. 2024HWYQ-009,
No. ZR2024MA007, No. tsqn202312060). Kejia Zhu is supported by NSFC of China (grant No. 12501086).\\

\section{Preliminaries}
We first list some basic definitions, notations, and facts that will be used later. Throughout the paper, $k$ denotes a number field and $\mathcal O_k$ its ring of integers. A \emph{place} of $k$ is an equivalence class of nontrivial absolute values on $k$, where two absolute values are
equivalent if they induce the same topology on $k$. The places of $k$ are divided into \emph{archimedean places} and \emph{finite places}. Write $V_\infty(k)$ and $V_f(k)$ for its sets of archimedean and finite places, respectively. For each place $v$, let $k_v$ be the completion of $k$ at $v$.
A real place corresponds to an embedding $k\hookrightarrow\mathbb R$, and a complex place to a conjugate pair of embeddings $k\hookrightarrow\mathbb C$; the corresponding completions are $\mathbb R$ and $\mathbb C$, respectively.
Finite places correspond to nonzero prime ideals of $\mathcal O_k$.
If $v$ lies above a rational prime $p$, then $k_v$ is a finite
extension of $\mathbb Q_p$.

For a linear algebraic $k$-group $\mathbf G$, write $\mathbf G_v:=\mathbf G\times_k k_v$ for its base change to $k_v$, and $\mathbf G(k_v)=\mathbf G_v(k_v)$ for its group of $k_v$-rational points, equipped with its natural topology.
We say that two algebraic $k$-groups $\mathbf G$ and $\mathbf H$
are isomorphic at $v$ if
$\mathbf G_v\cong_{k_v}\mathbf H_v$.

Using the identification
$k\otimes_{\mathbb Q}\mathbb R
 \cong\prod_{v\in V_\infty(k)}k_v$,
we have
\[
    \mathbf G(k\otimes_{\mathbb Q}\mathbb R)
    \cong
    \prod_{v\in V_\infty(k)}\mathbf G(k_v).
\]

\begin{definition}
Let $\mathbf{G}$ be a connected reductive algebraic group over a number field $k$, and let $v$ be a place of $k$. We say that $\mathbf{G}$ is \emph{$k_v$-anisotropic} at $v$ if $\operatorname{rank}_{k_v}\mathbf G=0$; otherwise, $\mathbf G$ is said to be \emph{$k_v$-isotropic} at $v$.
\end{definition}

\begin{definition}
    For a linear algebraic $k$-group $\mathbf{G}$, fix a faithful $k$-representation $\rho: \mathbf{G}\to \mathrm{GL}_N$. Define 
    \[
    \mathbf{G}(\mathcal{O}_k)\coloneqq \mathbf{G}(k)\cap \rho^{-1}(\mathrm{GL}(N,\mathcal{O}_k)).
    \]
    We say a subgroup $\Gamma<\mathbf{G}(k)$ is \emph{arithmetic} if $\Gamma$ is commensurable with $\mathbf{G}(\mathcal{O}_k)$. The defined commensurability is independent of the choice of the faithful $k$-representation $\rho$.
\end{definition}

\subsection{Congruence subgroup property and the finite adelic groups}
Now we list some results from \cite{kammeyer2023profinite} regarding the congruence subgroup property and the finite adelic groups that we will cite later in our paper. We first state a theorem originally from Borel and Harish-Chandra \cite{borel1962arithmetic} but reformulated by Kammeyer and Kionke:
\begin{theorem} \cite[Theorem 2.5(a)]{kammeyer2023profinite}\label{KK2.5a}
Let $k$ be a number field, let $H$ be a simply connected
simple linear algebraic group over $k$, and let
$\Gamma\leq H(k)$ be an arithmetic subgroup. Assume that $H_\infty:=
\prod_{v\in V_\infty(k)}H(k_v)$ is not compact, where $V_\infty(k)$ is the set of archimedean places of $k$, then $\Gamma$ is a lattice in
$H_\infty$. Moreover, if $H(k_v)$ is compact for some
archimedean place $v$, then $\Gamma$ is cocompact.
\end{theorem}

\begin{definition}
    For a finite place $v$, let $\mathcal O_v\subset k_v$ be its valuation ring.  The ring of finite adeles is the restricted product $$\A_k^f:=\prod_{v\in V_f(k)}' k_v
 =\left\{(x_v):x_v\in k_v\text{ for every }v,
 \quad x_v\in\mathcal O_v\text{ for all but finitely many }v\right\}.$$
\end{definition}

\begin{lemma}\cite[Lemma 2.6]{kammeyer2023profinite}\label{kklem2.6}
Let $k$ be an algebraic number field and let $G$, $H$ be two
semi-simple linear algebraic groups over $k$. If $G$ and $H$ are
isomorphic at all finite places, i.e.,
$G\times_k k_v \cong H\times_k k_v$
for all  $v\in V_f(k)$, then
\[
G(\mathbb A_k^f)
\cong
H(\mathbb A_k^f)
\]
as topological groups.
\end{lemma}

In the following definitions, let $k$ be a number field and $\mathbf{G}$ be a simply connected, absolutely almost simple algebraic group $k$-group, and let $\Gamma$ be an arithmetic subgroup of $\mathbf{G}(k)$. 
\begin{definition}
    We define the \emph{profinite completion} of $\Gamma$ by
    \[
    \widehat{\Gamma}\coloneqq \varprojlim_{N\in\mathcal{N}(\Gamma)} \Gamma/N,
    \]
    where $\mathcal{N}(\Gamma)$ denotes the set of finite-index normal subgroups of $\Gamma$.

    We define the \emph{congruence completion} of $\Gamma$ by
    \[
    \overline{\Gamma}\coloneqq \varprojlim_{M\in\mathcal{C}(\Gamma)} \Gamma/M,
    \]
    where $\mathcal{C}(\Gamma)$ denotes the set of normal congruence subgroups of $\Gamma$. 
    
    The inclusion $\mathcal{C}(\Gamma)\subset\mathcal{N}(\Gamma)$ induces a canonical homomorphism $\widehat{\Gamma}\to \overline{\Gamma}$.
\end{definition}
\begin{definition}
    We say that $\Gamma$ has the \emph{congruence subgroup property} (CSP) if the canonical homomorphism,
    \[
    \widehat{\Gamma}\to \overline{\Gamma},
    \]
    from the profinite completion to the congruence completion has a finite kernel.
    
    We say that $\mathbf{G}(k)$ has the CSP if some, equivalently every, arithmetic subgroup $\Gamma<\mathbf{G}(k)$ has the congruence subgroup property.
\end{definition}

\begin{theorem}{\cite[Theorem 2.4]{kammeyer2023profinite}}\label{congruencePro}
Let $H$ be a simply connected, absolutely almost simple algebraic
group over a number field $k$. Suppose that $H$ has the absolute type:\\
$B_\ell$ with $\ell\geq 2$,
$C_\ell$ with $\ell\geq 2$,
$D_\ell$ with $\ell\geq 5$,
$E_7,E_8, F_4,G_2$, that $k$ is not totally imaginary, and
that
\[
\sum_{v\in V_\infty(k)}
\operatorname{rank}_{k_v} H\geq 2.
\]
Then $H$ has the congruence subgroup property.
\end{theorem}

We next use the adelic consequence of CSP:

\begin{theorem}{\cite[Theorem 2.5 (b)]{kammeyer2023profinite}}\label{CSP}
Let $H$ be a simply connected simple algebraic group over a number field $k$, and let $\Delta\leq H(k)$ be an arithmetic subgroup.
Suppose that
\[
H_\infty:=\prod_{v\in V_\infty(k)}H(k_v)
\]
is noncompact and that $H$ has CSP. Then there exists an open subgroup $V$ of the profinite completion
$\widehat{\Delta}$ and a compact open subgroup
$K\leq H(\mathbb A_k^f)$, and a topological isomorphism $V\cong K$.
\end{theorem} 

\subsection{Consequences of Profinitely Isomorphic Lattices}

It is known by Kammeyer–Kionke that profinitely isomorphic lattices have adelic superrigidity:
\begin{theorem}\cite[Theorem 3.4]{kammeyer2021adelic}\label{thm:adelic}
    Let $\mathbf{G}$ and $\mathbf{H}$ be simply connected and absolutely almost simple groups of rank $\geq 2$, respectively over $k$ and $l$. Suppose there are arithmetic subgroups $\Gamma<\mathbf{G}(k)$ and $\Gamma'<\mathbf{H}(l)$ with $\widehat{\Gamma}\cong \widehat{\Gamma}'$. Then there exists an adeles isomorphism $j: \mathbb{A}^f_l\to \mathbb{A}^f_k$ together with an isomorphism  $\mathbf{G}\times_k\mathbb{A}^f_k\to \mathbf{H}\times_l\mathbb{A}^f_l$ over $j$.
\end{theorem}
When $k,l = \mathbb{Q}$, this implies that $\mathbf{G}$ and $\mathbf{H}$ agree at every finite place $p$, and one has a characterization from the symmetric space of their real places:
\begin{proposition}\cite[Proposition 2.5]{kammeyer2020profinite}\label{prop:symm_space}
    Let $\mathbf{G}$ and $\mathbf{G}'$ be semisimple linear algebraic groups over $\mathbb{Q}$, such that $\mathrm{Lie}(\mathbf{G})(\mathbb{Q}_p)\cong \mathrm{Lie}(\mathbf{G}')(\mathbb{Q}_p)$ for every prime $p$. Then their associated symmetric spaces $X = \mathbf{G}(\mathbb{R})/K$ and $X' = \mathbf{G}'(\mathbb{R})/K'$ satisfies that $\dim(X)\equiv \dim(X')\bmod 4$.
\end{proposition}

Echtler and Kammeyer provide another necessary condition for profinitely isomorphic lattices in terms of their real places.
\begin{definition}
Let $\mathbf{G}$ and $\mathbf{H}$ be algebraic groups over a field $F$ which become isomorphic over the algebraic closure $\overline F$. They are said to be \emph{inner forms} of one another if there exists an isomorphism $\varphi:\mathbf G_{\overline F}\xrightarrow{\sim}\mathbf H_{\overline F}$ such that for every $\sigma\in\operatorname{Gal}(\overline F/F)$, the automorphism $\varphi^{-1}\circ {}^\sigma\!\varphi: \mathbf G_{\overline F}\to \mathbf G_{\overline F}$ is in $\operatorname{Inn}(G_{\overline{F}})$.
\end{definition}
\begin{proposition}\cite[Proposition 8]{echtler2024bounded}\label{inner_twist}
    Let $\mathbf{G}$ and $\mathbf{H}$ be simply connected and absolutely almost simple groups of rank $\geq 2$, respectively over $k$ and $l$. Suppose $\mathbf{G}(k)$ has a unique isotropic archimedean place $v$, and $\mathbf{H}(l)$ has a unique isotropic archimedean place $w$. If there are arithmetic subgroups $\Gamma<\mathbf{G}(k)$ and $\Gamma'<\mathbf{H}(l)$ with $\widehat{\Gamma}\cong \widehat{\Gamma'}$, then $\mathbf{G}(k_v)$ and $\mathbf{H}(l_w)$ are inner twists of each other.
\end{proposition}
\subsection{Central simple algebras and Brauer classes}
Let $A$ be a \emph{central simple algebra} over $k$, that is a finite-dimensional associative simple $k$-algebra with center $Z(A) = k$. In particular, a $4$-dimensional central simple algebra $A/k$ is isomorphic to a quaternion algebra over $k$. 

By the Artin--Wedderburn theorem, there exists a unique central division algebra $D/k$, up to isomorphism, and a positive integer $r$ such that $A\cong M_r(D)$. 
\begin{definition}
Two central simple algebras $A/k$, $A'/k$ are \emph{Brauer equivalent} if their underlying central division algebras $D/k$ and $D'/k$ are isomorphic. We denote by
$[A]$ the \emph{Brauer class} of $A$.

The set of Brauer classes forms a group, under the tensor product operation $([A_1],[A_2])\to [A_1\otimes_k A_2]$, called the \emph{Brauer group} of $k$ and is denoted by $\operatorname{Br}(k)$. 
\end{definition}

For a place $v$ of $k$, the local Brauer group $\operatorname{Br}(k_v)$ is isomorphic to $\mathbb{Q}/\mathbb{Z}$ if $v$ is non-archimedean, is isomorphic to $\frac12\mathbb{Z}/\mathbb{Z}\subset\mathbb{Q}/\mathbb{Z}$ if $v$ is real, and is trivial if $v$ is complex.
\begin{theorem}[Albert--Brauer--Hasse--Noether]
    For a place $v$ of $k$, denote the corresponding local invariant map by $\operatorname{inv}_v:\operatorname{Br}(k_v)\to\mathbb{Q}/\mathbb{Z}$. Then for every central simple algebra $A/k$, one has
    \[
        \sum_v \operatorname{inv}_v([A\otimes_k k_v])=0\in \mathbb{Q}/\mathbb{Z}.
    \]
\end{theorem}

\section{On the Profinite Rigidity of Simple Lie Group Lattices}\label{section3}
In this section, we prove Theorem \ref{Mainthem} (see Propositions \ref{prop:classic}, \ref{prop:exceptional}, and \ref{prop:slh}, and Theorem \ref{thm:rigidity} for more details). By Theorem \ref{thm:adelic}, profinitely isomorphic arithmetic lattices in algebraic groups of rank $\geq 2$ induce a bijection between the finite places of their number fields, under which the corresponding local fields and algebraic groups are isomorphic. After base change to algebraic closures, these local isomorphisms yield an isomorphism of the absolute root systems of the ambient algebraic groups. On the other hand, we are concerned with whether different Lie groups of the same absolute Lie type can possess lattices with isomorphic profinite completions. 

\subsection{An exact isomorphism after passing to finite index}

The following two facts (see Lemma \ref{Experts1} and Lemma \ref{Experts2}) are standard, we include the proof for the sake of completeness:

\begin{lemma}\label{Experts1}
	Let $\Gamma$ be a residually finite group and let $H\leq\Gamma$ be a finite-index subgroup, then the closure $\overline H$ of $H$ in $\widehat{\Gamma}$ is open and that $\Gamma\cap\overline H=H$.
\end{lemma}
\begin{proof}
Define $C:=\bigcap_{\gamma\in\Gamma}\gamma H\gamma^{-1}$. Note $C$ is actually the kernel of the action of $\Gamma$ on the finite set $\Gamma/H$, thus $C$ is a finite-index normal subgroup of $\Gamma$. Since $\Gamma/C$ is finite, it is a discrete profinite group. Now by the universal property of profinite groups, $q:\Gamma\longrightarrow\Gamma/C$ extends uniquely to a continuous surjection $\widehat q:\widehat{\Gamma}\longrightarrow\Gamma/C$. Moreover, $W_H:=\widehat q^{-1}(H/C)$ is an open and closed subgroup of $\widehat{\Gamma}$. It's easy to see $\Gamma\cap W_H=q^{-1}(H/C)=H$. Because $\Gamma$ is dense in $\widehat{\Gamma}$ and $W_H$ is open, the intersection $\Gamma\cap W_H=H$ is dense in $W_H$. On the other hand, $W_H$ is closed and contains $H$. It follows that $\overline H=W_H$. So we can conclude that $\overline H=W_H$ is open in $\widehat{\Gamma}$ and $\Gamma\cap\overline H=\Gamma\cap W_H=H$.
\end{proof}

\begin{lemma}\label{Experts2}
    Let $\Gamma$ be a residually finite group and $V\subset \widehat\Gamma$ is open, then $G:=\Gamma\cap V$ has finite index in $\Gamma$.
\end{lemma}
\begin{proof}
    Since $V$ is open in the compact (topological) group $\widehat{\Gamma}$, it has finite index. Consider the map of coset spaces $$\phi:\Gamma/G\longrightarrow\widehat{\Gamma}/V,\qquad 
	\gamma G\longmapsto\gamma V.$$
	Note $\phi$ is injective: if $\gamma_1V=\gamma_2V$, then
	$\gamma_2^{-1}\gamma_1\in V\cap\Gamma=G,$
	and hence $\gamma_1G=\gamma_2G$. $\phi$ is also surjective. Indeed, every coset $xV$ is open in $\widehat{\Gamma}$, so the density of $\Gamma$ implies that $xV$ contains some $\gamma\in\Gamma$. Then
	$xV=\gamma V$.
	Thus $[\Gamma:G]=[\widehat{\Gamma}:V]<\infty$.
\end{proof}

\begin{proposition}\label{propo:exact}
   Let $\Gamma_1$ and $\Gamma_2$ be residually finite groups.  Suppose that there are isomorphic open subgroups $U_i\leq\widehat{\Gamma}_i$.  If each $\Gamma_i$ has a torsion-free, finite-index subgroup, then there are torsion-free, finite-index subgroups $G_i\leq\Gamma_i$ such that $\widehat{G}_1\cong\widehat{G}_2$. 
\end{proposition}
\begin{proof}
Let $\varphi:U_1\to U_2$ be an isomorphism.  Choose torsion-free finite-index subgroups $T_i\leq\Gamma_i$ and denote their closures in $\widehat{\Gamma}_i$ by $\overline T_i$.  Since $T_i$ has finite index in $\Gamma_i$ (thus also residually finite), by Lemma \ref{Experts1}, $\overline T_i$ is open in $\widehat{\Gamma}_i$ and $\Gamma_i\cap\overline T_i=T_i$.
Set
\[
V_1=U_1\cap\overline T_1\cap
\varphi^{-1}(U_2\cap\overline T_2),
\qquad
V_2=\varphi(V_1).
\]
Then $V_i$ is open in $\widehat\Gamma_i$ (hence $V_i$ is profinite) and $\varphi$ restricts to an isomorphism $V_1\cong V_2$.  Define $G_i:=\Gamma_i\cap V_i$.
By Lemma \ref{Experts2}, each $G_i$ has finite index in $\Gamma_i$. Also note $G_i\leq T_i$, hence it is torsion-free. Moreover, $\Gamma_i$ is dense in $\widehat{\Gamma}_i$ and $V_i$ is open, $G_i=\Gamma_i\cap V_i$ is dense in $V_i$, thus $\overline{G_i}=V_i$.

Now since $G_i=\Gamma_i\cap V_i$ has finite index in $\Gamma_i$, it's easy to see that the profinite topology on $G_i$ induced from $\Gamma_i$ (and thus from $V_i$) is the full profinite topology of $G_i$. 
Consequently, completion of the inclusion to the profinite group $V_i$:$G_i\hookrightarrow V_i$ gives a topological inclusion: $\widehat{G}_i=\overline G_i\hookrightarrow V_i$, which is actually also a surjection since $G_i$ is dense in $V_i$, thus we get a topological isomorphism $\widehat{G}_i\cong V_i$. Therefore $\widehat{G}_1\cong V_1\cong V_2\cong\widehat{G}_2$.
\end{proof}

\begin{lemma}\label{lem:CSP}
    Let $k$ be a number field, and let $\mathbf{G}_i$, $i=1,2$ be two non-isomorphic simply connected and absolutely almost simple $k$-groups. Suppose that for $i=1,2$, $\mathbf{G}_i(k)$ has the \emph{congruence subgroup property}, $G_{i,\infty}:=\prod_{v\in V_\infty(k)}\mathbf G_i(k_v)$ is noncompact, and for any finite place $v\in V_f(k)$, the completions
    \[
    \mathbf{G}_1\times_kk_v\cong \mathbf{G}_2\times_kk_v.
    \]
    Then there exists torsion-free arithmetic lattices
    \[
    \Gamma_1<\mathbf{G}_1(k),\ \Gamma_2<\mathbf{G}_2(k),
    \]
    such that the profinite completions $\widehat{\Gamma}_1\cong \widehat{\Gamma}_2$.
\end{lemma}

\begin{proof}
    By assumption, for every finite place $v$ of $k$, $\mathbf{G}_1\times_kk_v\cong \mathbf{G}_2\times_kk_v$. By Lemma \ref{kklem2.6}, there is an isomorphism of topological groups $\Phi\colon \mathbf G_1(\mathbb A_k^f)\cong\mathbf G_2(\mathbb A_k^f)$. Let $\Delta_i\leq \mathbf G_i(k),i=1,2$ be arithmetic subgroups; for example, we can take $\Delta_i = \mathbf{G}_i(\mathcal{O}_k)$. Since $G_{i,\infty}$ is noncompact and $\mathbf G_i$ has the congruence subgroup property, by Theorem \ref{CSP}, there exists an open subgroup $V_i\leq \widehat{\Delta}_i$, and a compact open subgroup $C_i\leq \mathbf G_i(\mathbb A_k^f)$, such that there is a topological isomorphism $\theta_i\colon V_i\cong C_i$ for each $i=1,2$. 

    Define $L:=\Phi(C_1)\cap C_2$. Since $\Phi(C_1)$ and $C_2$ are compact open subgroups of $\mathbf G_2(\mathbb A_k^f)$, their intersection $L$ is a compact open subgroup. Moreover, $L$ is open in both $\Phi(C_1)$ and $C_2$. Define 
    $$U_1:= \theta_1^{-1}\bigl(\Phi^{-1}(L)\bigr) \leq V_1, \qquad U_2:=\theta_2^{-1}(L)\leq V_2.$$
    The subgroup $\Phi^{-1}(L)$ is open in $C_1$, and $L$ is open in $C_2$. It follows that $U_i$ is open in $V_i$, thus open in $\widehat{\Delta}_i$. The composition $\psi:=\theta_2^{-1}\circ\Phi\circ\theta_1$ restricts to a topological isomorphism $\psi\colon U_1\xrightarrow{\sim}U_2$. Thus $\widehat{\Delta}_1$ and $\widehat{\Delta}_2$ contain isomorphic open subgroups.

    The arithmetic groups $\Delta_1$ and $\Delta_2$  are residually finite by Mal'cev's theorem and have torsion-free finite-index subgroups by Selberg's lemma. Proposition \ref{propo:exact} therefore applies and gives torsion-free finite-index subgroups $\Gamma_i\leq\Delta_i, i=1,2$, such that $\widehat{\Gamma}_1 \cong\widehat{\Gamma}_2$. Since finite-index subgroups of arithmetic groups are arithmetic, each $\Gamma_i$ is an arithmetic subgroup of $\mathbf G_i(k)$.
\end{proof}
\begin{corollary}\label{cor:CSP}
    For $i=1,2$, write $G_{i,\infty} = \prod_{v\in V_\infty(k)}\mathbf{G}_i(k_v)$ as in Lemma \ref{lem:CSP} as a product of compact and non-compact factors:
    \[
    G_{i,\infty} = G_{i,\infty}^{\mathrm{nc}} \times G_{i,\infty}^{\mathrm{c}}.
    \]
    Then $\Gamma_i$ is isomorphic to a irreducible lattice $\overline{\Gamma}_i$ in $G_{i,\infty}^{\mathrm{nc}}$; in particular, $\widehat{\overline{\Gamma}_1}\cong \widehat{\overline{\Gamma}_2}$.

    Moreover, if $G_{i,\infty}^{\mathrm{c}}$ is non-trivial, then $\overline{\Gamma}_i$ is cocompact.
\end{corollary}
\begin{proof}
By the Theorem \ref{KK2.5a}, the noncompactness of $G_{i,\infty}$ implies that $\Delta_i$, and hence $\Gamma_i$, is a lattice in $G_{i,\infty}$. If $\mathbf G_i(k_v)$ is compact for some archimedean place $v$, then the same theorem implies that this lattice is cocompact.

Projection onto $G_{i,\infty}^{\mathrm{nc}}$ sends $\Gamma_i$ to a lattice. Its kernel $\Gamma_i\cap G_{i,\infty}^{\mathrm{c}}$ is finite, since $\Gamma_i$ is discrete and $G_{i,\infty}^{\mathrm{c}}$ is compact. Since $\Gamma_i$ is torsion-free, the kernel is trivial. Thus $\Gamma_i$ is naturally isomorphic to its image in $G_{i,\infty}^{\mathrm{nc}}$. 

{Since $\mathbf G_i$ is absolutely almost simple over $k$, the standard
irreducibility theorem for arithmetic lattices implies that the image of
$\Gamma_i$ in $G^{\mathrm{nc}}_{i,\infty}$ is irreducible; see
\cite[Section~7.2]{lubotzky2003subgroup}.}
\end{proof}
\subsection{Profinite Non-rigidity of Simple Lie Groups}
In the following sub-subsections, we construct families of pairs of simply connected absolutely almost simple $k$-groups $\mathbf G$ and $\mathbf G'$ of the same absolute Dynkin type $A_\ell$, $B_\ell$, $C_\ell$, or $D_\ell$ such that $\mathbf G_v\cong\mathbf G'_v$ for every finite place $v\in V_f(k)$, while $\mathbf G_w\not\cong\mathbf G'_w$ for some real place $w\in V_\infty(k)$. Applying Lemma~\ref{lem:CSP} and Corollary~\ref{cor:CSP}, these constructions yield arithmetic lattices in different archimedean forms with isomorphic profinite completions. Thus the profinite completion of an arithmetic lattice does not, in general, determine the archimedean form of its ambient algebraic group.

Throughout this subsection, unless otherwise stated, let $\alpha$ be an algebraic unit such that $k=\mathbb Q(\alpha)$ is totally real, with $r_+\geq 1$ real embeddings $v_1$,\dots, $v_{r_+}$ at which it is positive, and exactly one distinguished real embedding at which $\alpha$ is negative. The choice of $\alpha$ may vary between constructions.
\subsubsection{Classical Real Form Constructions}
\begin{proposition}\label{prop:classic}
    There exist torsion-free cocompact arithmetic lattices $\Gamma<G$, $\Gamma' < G'$ with profinite completions $\widehat{\Gamma}\cong \widehat{\Gamma}'$ satifying any of the following:
    \begin{itemize}
        \item $G = \operatorname{Spin}(m,n)$, $G' = \operatorname{Spin}(m',n')$, where $d = m + n = m' + n' \geq 8$, $m,n,m',n'\geq 2$, and $mn\equiv m'n'\bmod 4$;
        \item $G = \operatorname{SU}(m,n)$, $G' = \operatorname{SU}(m',n')$, where $d = m + n = m' + n' \geq 7$, $m,n,m',n'\geq 2$, and $mn\equiv m'n'\bmod 2$;
        \item $G = \operatorname{Sp}(m,n)$, $G' = \operatorname{Sp}(m',n')$, where $d = m + n = m' + n' \geq 6$ and $m,n,m',n'\geq 2$.
    \end{itemize}
\end{proposition}
\begin{remark}
The algebraic groups used in the three constructions above are simply connected and absolutely almost simple of classical absolute Dynkin type. If $d=m+n=m'+n'$, then the spin groups have absolute type $B_\ell$ when $d=2\ell+1$ and $D_\ell$ when $d=2\ell$; the special unitary groups have absolute type $A_{d-1}$; and the quaternionic unitary groups giving the real forms $\mathrm{Sp}(m,n)$ have absolute type $C_d$. These identifications, together with the dimensions of the associated
symmetric spaces, are summarized in Table \ref{tab:classical-real-forms} in Appendix \ref{appendix:a}. After projecting away the compact archimedean factors, the ambient Lie groups are precisely the groups denoted $G_{i,\infty}^{\mathrm{nc}}$ in Corollary~\ref{cor:CSP}.
\end{remark}
We first construct the defining forms and verify that the associated algebraic groups are isomorphic at every finite place:
\begin{lemma}\label{lem:classic}
    \begin{itemize}
        \item Suppose $m,n,m',n'\geq 2$, $m+n = m'+n' = d$, and $4\mid(m-m')$. Define
        \[
        q = \langle 1^{[m]},\alpha^{[n]}\rangle,\quad q' = \langle 1^{[m']},\alpha^{[n']}\rangle,
        \]
        and let $\mathbf{G} = \operatorname{Spin}(q)$ and $\mathbf{G}' = \operatorname{Spin}(q')$ be the associated spin $k$-groups. Then for every finite place $v$ of $k$,
        \[
        \mathbf{G}\times_k k_v\cong \mathbf{G}'\times_k k_v.
        \]
        \item Suppose a unit $a\in k$ satisfies that $\sigma_v(a)<0$ at every real place (e.g., $a = -1$), $E = k(\sqrt{a})$, $m,n,m',n'\geq 2$, $m+n = m'+n' = d$, and $2\mid(m-m')$. Define the $E/k$-Hermitian forms
        \[
        h = \langle 1^{[m]},\alpha^{[n]}\rangle,\quad h' = \langle 1^{[m']},\alpha^{[n']}\rangle,
        \]
        and let $\mathbf{G} = \operatorname{SU}(h,E/k)$ and $\mathbf{G}' = \operatorname{SU}(h',E/k)$. Then for every finite place $v$ of $k$,
        \[
        \mathbf{G}\times_k k_v\cong \mathbf{G}'\times_k k_v.
        \]
        \item Suppose units $a,b\in k$ satisfies that $\sigma_v(a),\sigma_v(b)<0$ at every real place (e.g., $a=b=-1$), $B = (a,b)_k$ be a quaternion algebra, $m,n,m',n'\geq 2$, and $m+n = m'+n' = d$. Define the quaternionic Hermitian forms
        \[
        h = \langle 1^{[m]},\alpha^{[n]}\rangle,\quad h' = \langle 1^{[m']},\alpha^{[n']}\rangle,
        \]
        and let $\mathbf{G} = \operatorname{Sp}(h,B/k)$ and $\mathbf{G}' = \operatorname{Sp}(h',B/k)$. Then for every finite place $v$ of $k$,
        \[
        \mathbf{G}\times_k k_v\cong \mathbf{G}'\times_k k_v.
        \]
    \end{itemize}
\end{lemma}
\begin{proof}
    \textbf{Case (1): spin groups.} By \cite[Chapter VI, Theorem 2.12]{lam2005introduction}, a nondegenerate quadratic form over a non-archimedean local field is determined up to isometry by its dimension, determinant square class, and Hasse invariant. Let $v$ be a finite place of $k$. For $q = \langle 1^{[m]},\alpha^{[n]}\rangle$, the determinant square class is
    \[
    [\det(q)]_v=[\alpha^n]_v\in k_v^\times/(k_v^\times)^2,
    \]
    and the Hasse invariant is
    \[
    \epsilon_v(q)=\prod_{i<j}(a_i,a_j)_v
    =(\alpha,\alpha)_v^{n(n-1)/2},
    \]
    because $(1,1)_v=(1,\alpha)_v=1$.

    Since $4\mid(m-m')$, $4\mid(n-n')$, hence $q$ and $q'$ have the same determinant square class and the same Hasse invariant at every finite place, so $q\otimes_k k_v\cong q'\otimes_k k_v$. Therefore, the associated spin groups are $k_v$-isomorphic.

    \textbf{Case (2): unitary groups.} If $v$ splits in $E/k$, then $E_v\cong k_v\times k_v$, and
    \[
    \mathrm{SU}(h,E_v/k_v)\cong\mathrm{SL}(d,k_v),
    \]
    independently of the Hermitian form $h$.

    Suppose that $v$ is nonsplit, so $E_v/k_v$ is a quadratic field extension. Hermitian forms over $E_v/k_v$ are classified by their dimension and determinant class in
    \[
    k_v^\times/N_{E_v/k_v}(E_v^\times).
    \]
    Since $2\mid(m-m')$, $n-n'$ is also even, so
    \[
    \frac{\det(h)}{\det(h')}=\alpha^{n-n'}
    \]
    is a square in $k_v^\times$, hence a norm from $E_v^\times$. Thus $h$ and $h'$ have the same local determinant invariant, and the corresponding special unitary groups are $k_v$-isomorphic.

    \textbf{Case (3): symplectic groups.} If $B_v$ is split, then $B_v\cong M_2(k_v)$, and Morita equivalence gives
    \[
    \mathrm{Sp}(h,B_v/k_v)\cong\mathrm{Sp}(2d,k_v),
    \]
    independently of $h$.

    If $B_v$ is division, the reduced norm map
    \[
    \operatorname{Nrd}:B_v^\times\longrightarrow k_v^\times
    \]
    is surjective. Hence, all nondegenerate Hermitian forms of a fixed dimension over $B_v$ with respect to the canonical involution are isometric. In particular $h\otimes_k k_v\cong h'\otimes_k k_v$, and the corresponding $k_v$-groups are isomorphic.
\end{proof}
The remaining spin cases require local similarities rather than isometries. In one of these cases, we also modify a diagonal coefficient by a totally positive element:
\begin{lemma}\label{lem:spin_d}
    Let $d\geq6$ be even, and let $m,n,m',n'\geq2$ satisfy
    \[
    m+n=m'+n'=d,
    \qquad
    m\equiv n\pmod4,
    \qquad
    m'\equiv n'\pmod4,
    \qquad
    m-m'\equiv2\pmod4.
    \]
    \begin{itemize}
        \item If $d/2$ is odd, there exist $\alpha$ so for the quadratic forms
        \[
        q = \langle 1^{[m]},\alpha^{[n]}\rangle,\quad q' = \langle 1^{[m']},\alpha^{[n']}\rangle,
        \]
        and $\mathbf{G} = \operatorname{Spin}(q)$ and $\mathbf{G}' = \operatorname{Spin}(q')$,
        \[
        \mathbf{G}\times_k k_v\cong \mathbf{G}'\times_k k_v,
        \]
        for every finite place $v$ of $k$.
        \item If $d/2$ is even, there exist $\alpha$ and $\beta$, $\beta$ is positive at every real place, so for the quadratic forms
        \[
        q = \langle 1^{[m-1]},\beta,\alpha^{[n]}\rangle,\quad q' = \langle 1^{[m'-1]},\beta,\alpha^{[n']}\rangle,
        \]
        and $\mathbf{G} = \operatorname{Spin}(q)$ and $\mathbf{G}' = \operatorname{Spin}(q')$,
        \[
        \mathbf{G}\times_k k_v\cong \mathbf{G}'\times_k k_v,
        \]
        for every finite place $v$ of $k$.
    \end{itemize}
\end{lemma}
\begin{proof}
    For both cases, we assume that the minimal polynomial of $\alpha$ is irreducible modulo $2$, so that $k$ has a unique place $v_2$ above $2$.

    \textbf{Case (1): $d/2$ is odd.} Assume in addition that $[k:\mathbb Q]$ is even. At every finite place not above $2$, an argument similar to Lemma~\ref{lem:classic} gives the required local isomorphism. At $v_2$, Hilbert reciprocity gives $(\alpha,\alpha)_{v_2}=-1$. For $\lambda\in k_{v_2}^\times$, bilinearity of the Hilbert symbol gives
    \[
    \begin{split}
        \epsilon_{v_2}(\lambda q)
        &= (\lambda,\lambda)_{v_2}^{m(m-1)/2}(\lambda,\lambda\alpha)_{v_2}^{mn}
           (\lambda\alpha,\lambda\alpha)_{v_2}^{n(n-1)/2} \\
        &= (\lambda,-\alpha)_{v_2}\,\epsilon_{v_2}(q).
    \end{split}
    \]
    Since $[k:\mathbb Q]$ is even and $\alpha$ is negative at exactly one real place, the number of real places at which $-\alpha$ is negative is odd. Hilbert reciprocity therefore gives
    \[
    (-\alpha,-\alpha)_{v_2}=-1,
    \]
    so $-\alpha$ is not a square in $k_{v_2}$. Hence there exists $\lambda\in k_{v_2}^\times$ with $(\lambda,-\alpha)_{v_2}=-1$. It follows that
    \[
    \epsilon_{v_2}(\lambda q)=-\epsilon_{v_2}(q)=\epsilon_{v_2}(q').
    \]
    Because $d$ is even, scaling by $\lambda$ does not change the determinant square class, and therefore $\lambda q\cong q'$ over $k_{v_2}$. Since scalar multiples define the same special orthogonal and spin groups, the two spin groups are locally isomorphic at $v_2$ as required.

    \textbf{Case (2): $d/2$ is even.} Assume $\beta$ is not a square in $k_{v_2}$. Such a choice of $\alpha$ and $\beta$ is possible; for example, take $\alpha$ to be a root of
    \[
    f(x)=x^3-3x^2-2x+1
    \]
    and put $\beta=\alpha+1$. The reduction of $f$ modulo $2$ is irreducible. A direct comparison modulo $4$, using $\alpha^3=3\alpha^2+2\alpha-1$, shows that $\beta$ is not a square in $k_{v_2}$.

    Since $\beta$ is totally positive, these forms have signatures $(m,n)$ and $(m',n')$ at the distinguished real place and are positive definite at every other real place. Their determinant square classes agree at every finite place. For a finite place $v\nmid2$,
    \[
    \epsilon_v(q)=(\alpha,\beta)_v^n(\alpha,\alpha)_v^{n(n-1)/2},
    \qquad
    \epsilon_v(q')=(\alpha,\beta)_v^{n'}(\alpha,\alpha)_v^{n'(n'-1)/2}.
    \]
    Here $(\alpha,\alpha)_v=1$, and $n-n'$ is even, so the Hasse invariants agree. At $v_2$, for $\lambda\in k_{v_2}^\times$ one similarly obtains
    \[
    \epsilon_{v_2}(\lambda q)=(\lambda,\beta)_{v_2}\epsilon_{v_2}(q)
    =-(\lambda,\beta)_{v_2}\epsilon_{v_2}(q').
    \]
    Since $\beta$ is not a square, choose $\lambda$ with $(\lambda,\beta)_{v_2}=-1$. Then $\lambda q\cong q'$ over $k_{v_2}$, so the associated spin groups are locally isomorphic at every finite place.
\end{proof}
\begin{proof}[Proof of Proposition \ref{prop:classic}]
    Since $m+n = m'+n' = d$, we have
    \[
    mn - m'n' = m(d-m) - m'(d-m') = (m-m')(d-m-m') = (m-m')(n-m').
    \]
    Up to interchanging $m$ and $n$, which preserves the isomorphic type of $G$, the modulo condition in case (1) implies either $4\mid (m-m')$, or $m-m'\equiv n-m'\equiv 2\bmod 4$. The modulo condition in case (2) implies $2\mid (m-m')$.

    \textbf{Case (1): spin groups.} In the two subcases, Lemma \ref{lem:classic}, case (1), and Lemma \ref{lem:spin_d} provide algebraic groups $\mathbf{G}$ and $\mathbf{G'}$ that are isomorphic at every finite place. As spin groups of quadratic forms of dimension at least $5$ and have real rank $\geq 2$, Kneser's congruence subgroup theorem \cite{GDZPPN002196646} implies that both $\mathbf G$ and $\mathbf G'$ have CSP. Lemma~\ref{lem:CSP} therefore yields torsion-free arithmetic subgroups
    \[
    \widetilde\Gamma<\mathbf G(k),
    \qquad
    \widetilde\Gamma'<\mathbf G'(k),
    \]
    with isomorphic profinite completions.

    By the Borel--Harish-Chandra theorem (Theorem~\ref{KK2.5a}), $\widetilde\Gamma$ is a lattice in
    \[
    \mathbf G(k\otimes_\mathbb Q\mathbb R)
    \cong\prod_{v\in V_\infty(k)}\mathrm{Spin}(q\otimes_k k_v).
    \]
    At a real place $v$ with $\sigma_v(\alpha)>0$, the form $q$ is positive definite and the corresponding factor is $\mathrm{Spin}(d)$. At the distinguished real place with $\sigma_v(\alpha)<0$, the factor is $\mathrm{Spin}(m,n)$. Counting the number of different type of real places, we obtain
    \[
    \mathbf G(k\otimes_\mathbb Q\mathbb R)
    \cong
    \mathrm{Spin}(d)^{r_+}\times
    \mathrm{Spin}(m,n).
    \]
    Project away from the compact factor $\mathrm{Spin}(d)^{r_+}$ and put
    \[
    \Gamma=\operatorname{pr}(\widetilde\Gamma).
    \]
    The kernel of $\operatorname{pr}|_{\widetilde\Gamma}$ is finite because it is a discrete subgroup of a compact group. Since $\widetilde\Gamma$ is torsion-free, this kernel is trivial. Thus $\Gamma\cong\widetilde\Gamma$. Because $r_+\geq 1$, Corollary \ref{cor:CSP} gives cocompactness and irreducibility. The same argument for $\mathbf G'$ produces $\Gamma'\cong\widetilde\Gamma'$. Therefore $\widehat\Gamma\cong\widehat{\Gamma'}$.

    \textbf{Case (2): unitary groups.} Lemma \ref{lem:classic}, case (2) provides algebraic groups $\mathbf{G}$ and $\mathbf{G'}$ that are isomorphic at every finite place. The groups $\mathbf G$ and $\mathbf G'$ are simply connected absolutely almost simple outer $k$-groups of absolute type $A_{d-1}$. Since $k$ is not totally imaginary and the groups are assumed to have archimedean rank at least $2$, Kammeyer and Kionke's congruence subgroup theorem \cite[Theorem~2.4]{kammeyer2023profinite} applies to these special unitary groups of type ${}^2A_{d-1}$. Lemmas \ref{lem:CSP} therefore give torsion-free arithmetic lattices in $\mathbf G(k)$ and $\mathbf G'(k)$ with isomorphic profinite completions.

    Since $\sigma_v(a)<0$ at every archimedean place, the quadratic algebra $E_v\cong \mathbb{C}$. The archimedean factor in the Borel-Harish-Chandra theorem is $\mathrm{SU}(m,n)$ when $\sigma_v(\alpha)<0$ and the compact group $\mathrm{SU}(d)$ when $\sigma_v(\alpha)>0$. Consequently,
    \[
    \mathbf G(k\otimes_\mathbb Q\mathbb R)
    \cong
    \mathrm{SU}(d)^{r_+}\times
    \mathrm{SU}(m,n).
    \]
    The same formula holds for $\mathbf G'$ with $(m,n)$ replaced by $(m',n')$. Projecting away the compact $\mathrm{SU}(d)^{r_+}$ factors is injective on the torsion-free arithmetic subgroups and preserves irreducibility and cocompactness as claimed.

    \textbf{Case (3): symplectic groups.} Lemma \ref{lem:classic}, case (3) provides algebraic groups $\mathbf{G}$ and $\mathbf{G'}$ that are isomorphic at every finite place, with absolute type $C_d$. Since $k$ is not totally imaginary and the groups are assumed to have archimedean rank at least $2$, Kammeyer and Kionke's congruence subgroup theorem \cite[Theorem 2.4]{kammeyer2023profinite} applies to these groups of type $C_d$. Lemmas~\ref{lem:CSP} therefore give torsion-free arithmetic subgroups with isomorphic profinite completions.

    Since $a$ and $b$ are negative at every archimedean place $v$, $B_v$ is non-split and isomorphic to $\mathbb{H}$. The factor groups are $\mathrm{Sp}(m,n)$ or the compact form $\mathrm{Sp}(d)$, depending on the sign of $\sigma_v(\alpha)$. Thus
    \[
    \mathbf G(k\otimes_\mathbb Q\mathbb R)
    \cong
    \mathrm{Sp}(d)^{r_+}\times\mathrm{Sp}(m,n),
    \]
    and similarly for $\mathbf G'$. Projection away from the compact factors is injective on the torsion-free arithmetic lattices and preserves irreducibility and cocompactness, so the conclusion follows.
\end{proof}

\subsubsection{Tits Constructions}
We obtain the following pairs in the exceptional Lie algebra types:
\begin{proposition}\label{prop:exceptional}
    There exist torsion-free arithmetic lattices $\Gamma<G$, $\Gamma'<G'$ with isomorphic profinite completions if the simple Lie groups $G,G'$ satisfy any of the following:
    \begin{itemize}
        \item $G = E_{6(2)}$ and $G' = E_{6(-14)}$,
        \item $G = E_{7(7)}$ and $G' = E_{7(-25)}$.
        \item $G = E_{8(8)}$ and $G' = E_{8(-24)}$.
    \end{itemize}
    Here, $\Gamma$ and $\Gamma'$ are cocompact for the $E_7$ and $E_8$ cases.

    In addition, assuming the congruence subgroup property for anisotropic groups of type $E_6$, which remains open in general, we can also let $G = E_{6(6)}$ and $G' = E_{6(-26)}$; moreover, $\Gamma$ and $\Gamma'$ for both $E_6$ cases may be chosen cocompact.
\end{proposition}
For the exceptional real forms, we use Tits' construction from composition algebras and Albert algebras \cite{tits1965algebres}.
\begin{lemma}\label{lem:albert}
    In each of the following cases, there exists a real quadratic field $k=\mathbb Q(\alpha)$ with real places $v_1$, $v_2$, and two Albert algebras $A$, $A'$ over $k$, with the predicted real place identifications:
    \begin{itemize}
        \item $A_{v_1}$ is isomorphic to the split real Albert algebra, while $A'_{v_1}$, $A_{v_2}$ and $A'_{v_2}$ are all isomorphic to the compact one.
        \item $A_{v_1}$ is isomorphic to the split real Albert algebra, $A'_{v_1}$ is isomorphic to the intermediate one, while both $A_{v_2}$ and $A'_{v_2}$ are isomorphic to the compact one.
    \end{itemize}
\end{lemma}
\begin{proof}
    In this proof, we suppose that the two real places are such that $\sigma_{v_1}(\alpha)>0$ and $\sigma_{v_2}(\alpha)<0$. Take units $a,b,c'\in k$ that are negative at both real places, and a distinguished unit $c$ with $\sigma_{v_1}(c)>0$ and $\sigma_{v_2}<0$; for example $a=b=c'=-1$ and $c = \alpha$. Define the octonion algebras $C= (a,b,c)_k$ and $C'= (a,b,c')_k$, then $C$ and $C'$ satisfies that
    \[
    C_{v_1}\cong \mathbb{O}_s,\quad C'_{v_1}\cong C_{v_2}\cong C'_{v_2}\cong \mathbb{O}.
    \]

    (1) Consider the octonionic Hermitian form $h = \langle 1^{[3]}\rangle$, and Albert algebras
    \[
    A = H_3(C,h),\quad A' = H_3(C',h).
    \]
    Therefore,
    \[
    A_{v_1}\cong H_3(\mathbb{O}_s),\quad A'_{v_1}\cong A_{v_2}\cong A'_{v_2}\cong H_3(\mathbb{O}),
    \]
    which are split and compact, respectively.
    
    (2) Consider the octonionic Hermitian form $h = \langle 1^{[2]},-\alpha^{[1]}\rangle$, and Albert algebras
    \[
    A = H_3(C,h),\quad A' = H_3(C',h).
    \]
    The construction of $h$ satisfies that
    \[
    h_{v_1} \cong \langle 1^{[2]},-1^{[1]}\rangle,\quad h_{v_2} \cong \langle 1^{[3]}\rangle.
    \]
    Therefore,
    \[
    A_{v_1}\cong H_3(\mathbb{O}_s),\quad A'_{v_1}\cong H_3(\mathbb{O},(1,1,-1)),\quad A_{v_2}\cong A'_{v_2}\cong H_3(\mathbb{O}),
    \]
    which are, respectively, the split, intermediate, and compact real Albert algebras, as desired.
\end{proof}
\begin{remark}\label{rem:albert}
    When $k = \mathbb{Q}$, the construction $C/\mathbb{Q} = (-1,-1,1)_{\mathbb{Q}}$, $C'/\mathbb{Q} = (-1,-1,-1)_{\mathbb{Q}}$, and $h = \langle 1^{[2]},-1^{[1]}\rangle$ produces Albert algebras $A/\mathbb{Q}$ and $A'/\mathbb{Q}$ similar to case (2) at the real place, namely, $A_\infty$ is split while $A'_\infty$ is intermediate.
\end{remark}
\begin{proof}[Proof of Proposition \ref{prop:exceptional}]
    Take units $a,b,c\in k$ that satisfy the condition in the proof of Lemma~\ref{lem:albert}, e.g., $a = b=-1$ and $c = \alpha$.
    
    \textbf{Case (1): $G$ is split.} Take the quadratic field extension $E= k(\sqrt{c})$, quaternion algebra $B= (b,c)_k$, octonion algebra $C = (a,b,c)_k$, and Albert algebras $A$, $A'$ in Lemma~\ref{lem:albert}, case (1). For the $E_6$ case, we let
    \[
    \mathbf{G} = G(A,E),\quad \mathbf{G}' = G(A',E).
    \]
    For the $E_7$ case, we let
    \[
    \mathbf{G} = G(A,B),\quad \mathbf{G}' = G(A',B).
    \]
    For the $E_8$ case, we let
    \[
    \mathbf{G} = G(A,C),\quad \mathbf{G}' = G(A',C).
    \]
    At each finite place $w$ of $k$, there is a single isomorphism type of Albert algebra over $k_w$ (see, e.g., \cite{schwermer2011geometric}), thus $\mathbf{G}_w\cong \mathbf{G}'_w$. On the real places, we have
    \[
    E_{v_1}\cong\mathbb R\times\mathbb R,\quad E_{v_2}\cong\mathbb C,\quad B_{v_1}\cong M_2(\mathbb R),\quad B_{v_2}\cong\mathbb H,\quad C_{v_1}\cong\mathbb{O}_s,\quad C_{v_2}\cong\mathbb{O}.
    \]
    Using the identification of the real forms arising from Tits's magic-square construction in \cite[Sections 3 and 7]{barton2003magic}, in the three cases of $\mathbf{G}$ and $\mathbf{G}'$, we have
    \[
    \mathbf{G}_{v_1}\cong E_{6(6)},\quad \mathbf{G}'_{v_1}\cong E_{6(-26)},\quad \mathbf{G}_{v_2}\cong\mathbf{G}'_{v_2}\cong E_{6(-78)},
    \]
    or
    \[
    \mathbf{G}_{v_1}\cong E_{7(7)},\quad \mathbf{G}'_{v_1}\cong E_{7(-25)},\quad \mathbf{G}_{v_2}\cong\mathbf{G}'_{v_2}\cong E_{7(-133)},
    \]
    or
    \[
    \mathbf{G}_{v_1}\cong E_{8(8)},\quad \mathbf{G}'_{v_1}\cong E_{8(-24)},\quad \mathbf{G}_{v_2}\cong\mathbf{G}'_{v_2}\cong E_{8(-248)}.
    \]
    For the $E_7$ and $E_8$ groups, the congruence subgroup property is known, so Lemma~\ref{lem:CSP} and Corollary~\ref{cor:CSP} give torsion-free cocompact lattices with isomorphic profinite completions. The same conclusion holds for the $E_6$ pair under the stated assumption on the congruence subgroup property for anisotropic groups of type $E_6$.
    
    \textbf{Case (2): $G$ is quasi-split.} Let $A$ and $A'$ be Albert algebras as in Lemma~\ref{lem:albert}, case (2), and quadratic extension $E = k(\sqrt{a})$. Tits' construction gives simply connected groups, \cite{garibaldi2007groups}
    \[
    \mathbf{G} = G(A,E),\quad \mathbf{G}' = G(A',E).
    \]
    Again, $H_3(C)$ and $H_3(C')$ are isomorphic at each finite place $w$, thus $\mathbf{G}_w\cong \mathbf{G}'_w$. For the two real places $v_1$ and $v_2$, $E_{v_1} \cong E_{v_2}\cong \mathbb{C}$ are fields, thus \cite[4.4]{garibaldi2007groups} implies that $\mathbf{G}_{v_i}$ and $\mathbf{G}'_{v_i}$ are of type $^2E_6$. In particular, together with the real forms of $A$ and $A'$ as described in Lemma~\ref{lem:albert}, the corresponding real forms are identified in \cite{jacobson2017exceptional}:
    \[
    \mathbf{G}_{v_1}\cong E_{6(2)},\quad \mathbf{G}'_{v_1}\cong E_{6(-14)},\quad \mathbf{G}_{v_2}\cong\mathbf{G}'_{v_2}\cong E_{6(-78)}.
    \]
    That is to say,
    \[
    \mathbf{G}(k\otimes_\mathbb{Q}\mathbb{R})\cong E_{6(2)}\times E_{6(-78)},\quad \mathbf{G}'(k\otimes_\mathbb{Q}\mathbb{R})\cong E_{6(-14)}\times E_{6(-78)}.
    \]
    The compact factor $E_{6(-78)}$ implies that $\mathbf{G}$ and $\mathbf{G}'$ are anisotropic over $k$. Assuming the congruence subgroup property for these anisotropic $k$-groups of type ${}^2E_6$, Lemma~\ref{lem:CSP} together with Corollary~\ref{cor:CSP} thus yields cocompact torsion-free lattices $\Gamma<E_{6(2)}$ and $\Gamma'<E_{6(-14)}$ with isomorphic profinite completions.

    Alternatively, consider $k = \mathbb{Q}$, $E=\mathbb{Q}(\sqrt{-1})$, and the Albert algebras $A/\mathbb{Q}$ and $A'/\mathbb{Q}$ resulting from Remark~\ref{rem:albert}. Since the Hermitian form $h=\langle1^{[2]},-1^{[1]}\rangle$ satisfies the condition in \cite[Proposition~6.1]{garibaldi2007groups}, the proposition implies that $\mathbb{Q}\times E$ is a subalgebra of both $A$ and $A'$. By \cite[Theorem~0.2]{garibaldi2007groups}, $\mathbf{G}=G(A,E)$ and $\mathbf{G}'=G(A',E)$ are isotropic. Therefore, by \cite[Theorem~2.4]{kammeyer2023profinite}, both satisfy the congruence subgroup property. Lemma~\ref{lem:CSP} and Corollary~\ref{cor:CSP} then yield profinitely isomorphic torsion-free non-cocompact lattices $\Gamma<E_{6(2)}$ and $\Gamma'<E_{6(-14)}$.
\end{proof}
\subsubsection{Outer Real Forms of Type \texorpdfstring{$A_{4n-1}$}{A4n-1}}
The remaining type $A$ pairs arise from unitary groups over algebras with involutions of the second kind:
\begin{proposition}\label{prop:slh}
    For every $n\geq2$, there exist torsion-free cocompact arithmetic lattices
    \[
    \Gamma<\mathrm{SL}(4n,\mathbb R),
    \qquad
    \Gamma'<\mathrm{SL}(2n,\mathbb H),
    \]
    with $\widehat\Gamma\cong\widehat{\Gamma'}$.
\end{proposition}

\begin{lemma}\label{lem:slh}
    Suppose additionally that $\alpha$ has degree $2$, and the minimal polynomial of $\alpha$ is irreducible modulo $2$. Let $E = k(\sqrt{a})$, where $a\in\mathcal{O}_k^\times$ such that $\sigma_{v_1}(a)>0$ and $\sigma_{v_2}(a)<0$, e.g., $a = \alpha$. Then the central simple $E$-algebras,
    \[
    A=M_2(E),
    \qquad
    A'=(-1,-1)_E,
    \]
    at the two real places $v_1,v_2$ of $k$ are
    \[
    A_{v_1}\cong M_2(\mathbb R)\times M_2(\mathbb R),
    \qquad
    A'_{v_1}\cong\mathbb H\times\mathbb H,
    \qquad
    A_{v_2}\cong A'_{v_2}\cong M_2(\mathbb C),
    \]
    while $A_w\cong A'_w$ for every finite place $w$ of $k$.
\end{lemma}
\begin{proof}
    The real places of $E$ are
    \[
    E_{v_1}\cong\mathbb R\times\mathbb R,
    \qquad
    E_{v_2}\cong\mathbb C.
    \]
    Hence the matrix and quaternionic $E$-algebras $A$ and $A'$ at the places $v_1$ and $v_2$ are
    \[
    A_{v_1}\cong M_2(\mathbb R)\times M_2(\mathbb R),
    \qquad
    A'_{v_1}\cong\mathbb H\times\mathbb H,
    \]
    and
    \[
    A_{v_2}\cong A'_{v_2}\cong M_2(\mathbb C).
    \]

    Let $w$ be a finite place of $k$. If $w$ lies above an odd prime $p$, then the Hamilton quaternion algebra $(-1,-1)_\mathbb Q$ splits over $\mathbb Q_p$, so
    \[
    A'_w\cong M_2(E_w)\cong A_w.
    \]
    Above $2$, irreducibility of the minimal polynomial modulo $2$ gives a unique place $w$ with $[k_w:\mathbb Q_2]=2$. The local invariant of $(-1,-1)_\mathbb Q$ at $2$ is $1/2$, hence after restriction to $k_w$ it becomes
    \[
    [k_w:\mathbb Q_2]\cdot\frac12=0\in\mathbb Q/\mathbb Z.
    \]
    Thus $(-1,-1)_{k_w}$ is split, and again $A'_w\cong M_2(E_w)\cong A_w$.
\end{proof}

\begin{proof}[Proof of Proposition \ref{prop:slh}]
    Let $k,E,A,A'$ be as in Lemma~\ref{lem:slh}. Let $\overline{\phantom{x}}$ denote the nontrivial $k$-automorphism of $E$. Define unitary involutions
    \[
    \tau:A\longrightarrow A,
    \qquad
    \tau(X)=\overline X^{\mathsf T},
    \]
    and
    \[
    \tau':A'\longrightarrow A',
    \qquad
    \tau'(x_0+x_1\mathrm i+x_2\mathrm j+x_3\mathrm{ij})
    =\overline{x_0}-\overline{x_1}\mathrm i-\overline{x_2}\mathrm j-\overline{x_3}\mathrm{ij}.
    \]
    Using the standard Hermitian form $h=\langle1^{[2n]}\rangle$, put
    \[
    \mathbf G=\mathrm{SU}(2n,A,\tau),
    \qquad
    \mathbf G'=\mathrm{SU}(2n,A',\tau').
    \]

    At $v_1$, the involutions identify with
    \[
    \tau:(X,Y)\longmapsto(Y^{\mathsf T},X^{\mathsf T}),
    \qquad
    \tau':(x,y)\longmapsto(y^*, x^*),
    \]
    and therefore
    \[
    \mathbf G(k_{v_1})\cong\mathrm{SL}(4n,\mathbb R),
    \qquad
    \mathbf G'(k_{v_1})\cong\mathrm{SL}(2n,\mathbb H).
    \]
    At $v_2$, both involutions become conjugate transpose on $M_2(\mathbb C)$, so
    \[
    \mathbf G(k_{v_2})\cong\mathbf G'(k_{v_2})\cong\mathrm{SU}(4n),
    \]
    which is compact.

    Let $w$ be a finite place. If $E/k$ splits at $w$, then
    \[
    A_w\cong A'_w\cong M_2(k_w)\times M_2(k_w),
    \]
    and both local groups are isomorphic to $\mathrm{SL}(4n,k_w)$. If $E_w/k_w$ is a field extension, Morita equivalence identifies
    \[
    \mathrm{SU}(2n,A_w,\tau_w)\cong\mathrm{SU}(4n,E_w/k_w,h_w),
    \]
    and
    \[
    \mathrm{SU}(2n,A'_w,\tau'_w)\cong\mathrm{SU}(4n,E_w/k_w,h'_w),
    \]
    where $h_w\cong\langle1^{[4n]}\rangle$ and $h'_w\cong h_0^{\perp 2n}$ for a two-dimensional Hermitian form $h_0$. Hence
    \[
    [\det(h'_w)]=[\det(h_0)^{2n}]=[1]=[\det(h_w)]
    \quad\text{in}\quad
    k_w^\times/N_{E_w/k_w}(E_w^\times).
    \]
    The local Hermitian forms therefore have the same determinant invariant, and the corresponding special unitary groups are $k_w$-isomorphic. Thus $\mathbf G$ and $\mathbf G'$ are isomorphic at every finite place.

    Via Morita equivalence, $\mathbf G$ is a special unitary group attached to a $4n$-dimensional Hermitian form over $E/k$, the congruence subgroup property follows from \cite{kammeyer2023profinite}. For $\mathbf G'$, the group is a special unitary group of Hermitian dimension $2n$ over the quaternion division algebra $A'=(-1,-1)_E$, with respect to the involution $\tau'$ of the second kind. Since it has Hermitian dimension $2n\geq 4$ and real rank $2n-1> 2$, Tomanov's congruence subgroup theorem for such quaternionic unitary groups applies; see \cite[Theorem 3 and the discussion in page 178]{Rap99}. Lemma~\ref{lem:CSP} and Corollary~\ref{cor:CSP}, together with the compact factor at $v_2$, now give torsion-free cocompact lattices
    \[
    \Gamma<\mathrm{SL}(4n,\mathbb R),
    \qquad
    \Gamma'<\mathrm{SL}(2n,\mathbb H),
    \]
    with isomorphic profinite completions.
\end{proof}
\subsection{Classification of Profinite Rigidity across Real Forms}
Let $G$ and $G'$ be non-isomorphic simply connected simple real Lie groups of real rank at least $2$ and of the same absolute Dynkin type. We now classify when they can contain torsion-free lattices with isomorphic profinite completions.

\begin{theorem}\label{thm:rigidity}
    Let $G$ and $G'$ be non-isomorphic simple real Lie groups of the same absolute Dynkin type. Also, suppose they are simply connected, with real ranks at least $2$.
    
    \begin{enumerate}
        \item If, up to interchanging $G$ and $G'$, the pair falls into one of the following families, then there exist torsion-free lattices $\Gamma<G$ and $\Gamma'<G'$ such that $\widehat\Gamma\cong\widehat{\Gamma'}$:
        \begin{itemize}
            \item $G=\mathrm{Spin}(m,n)$ and $G'=\mathrm{Spin}(m',n')$, where $m+n=m'+n'\geq8$ and
            \[
            mn\equiv m'n'\pmod4;
            \]
            \item $G=\mathrm{SU}(m,n)$ and $G'=\mathrm{SU}(m',n')$, where $m+n=m'+n'\geq7$ and
            \[
            mn\equiv m'n'\pmod2;
            \]
            \item $G=\mathrm{Sp}(m,n)$ and $G'=\mathrm{Sp}(m',n')$, where $m+n=m'+n'\geq6$;
            \item $G=\mathrm{SL}(4n,\mathbb R)$ and $G'=\mathrm{SL}(2n,\mathbb H)$, where $n\geq2$;
            \item $(G,G')=(E_{6(-14)},E_{6(2)})$, $(E_{7(-25)},E_{7(7)})$, or $(E_{8(-24)},E_{8(8)})$.
        \end{itemize}

        \item Assuming the congruence subgroup property for anisotropic groups of type $E_6$, the additional pair $(E_{6(-26)},E_{6(6)})$ also satisfies (1).

        \item In all other cases, lattices $\Gamma<G$ and $\Gamma'<G'$ have non-isomorphic profinite completions.
    \end{enumerate}
\end{theorem}
\begin{remark}\label{remk:rigidity}
    The $\operatorname{Spin}$, $\operatorname{SU}$ or $\operatorname{SL}$, and $\operatorname{Sp}$ groups occurring in Theorem \ref{thm:rigidity} have absolute Dynkin types $B_\ell$ or $D_\ell$, $A_\ell$, and $C_\ell$, respectively. The noncompact real forms of classical Lie algebras are summarized in Table \ref{tab:classical-real-forms} in Appendix \ref{appendix:a}.
   \end{remark}

Before proving the theorem, we exclude two families using their Brauer classes.

\begin{lemma}\label{lem:rigidity_brauer}
    Suppose either
    \[
    G\cong\mathrm{Sp}(m,n),
    \qquad
    G'\cong\mathrm{Sp}(2(m+n),\mathbb R),
    \]
    or
    \[
    G\cong\mathrm{Spin}(m,n),
    \qquad
    G'\cong\mathrm{Spin}^*(m+n),
    \]
    where in the second case $m+n$ is even and at least $10$. Assume that both groups have real rank at least $2$. Then for any lattices $\Gamma<G$ and $\Gamma'<G'$,
    \[
    \widehat\Gamma\not\cong\widehat{\Gamma'}.
    \]
\end{lemma}

\begin{proof}
Suppose to the opposite that $\widehat{\Gamma}\cong\widehat{\Gamma'}$. By Margulis arithmeticity, after replacing $\Gamma$, $\Gamma'$ by corresponding finite-index subgroups, we may assume that they are arithmetic subgroups of simply connected absolutely almost simple groups $\mathbf G/k$ and $\mathbf G'/l$, respectively, where $k$ and $l$ are totally real. There are distinguished real places $v$ of $k$ and $v'$ of $l$ such that
\[
\mathbf G(k_v)\cong G,\qquad \mathbf G'(l_{v'})\cong G',
\]
while the groups at all other real places are compact. By Theorem~\ref{thm:adelic}, the finite places of $k$ and $l$ are matched so that the corresponding local fields and algebraic groups are isomorphic. The adeles isomorphism also implies the degrees $[k:\mathbb{Q}] = [l:\mathbb{Q}]$.

In the symplectic case, let $\mathbf G_0$ be the split simply connected group of the corresponding type $C_d$. Following \cite[Proposition 11]{echtler2024bounded}, the $k$- and $l$-forms $\mathbf G$ and $\mathbf G'$ determine classes
\[
\alpha\in H^1(k,\operatorname{Ad}\mathbf G_0),
\qquad
\beta\in H^1(l,\operatorname{Ad}\mathbf G_0).
\]
For $F = k$ or $l$, the Kummer sequence $1\to\mu_2\to\mathbf G_0\to\operatorname{Ad}\mathbf G_0\to 1$ yields a boundary map $\delta_F:H^1(F,\operatorname{Ad}\mathbf G_0)\to H^2(F,\mu_2)$ to Tits classes, which are regarded as Brauer classes via the Kummer identification $H^2(F,\mu_2)\cong \operatorname{Br}(F)[2]$. At corresponding finite places, the localizations of $\delta_k(\alpha)$ and $\delta_l(\beta)$ agree by adelic superrigidity. Since $k/\mathbb{Q}$ and $l/\mathbb{Q}$ are totally real and have the same degree, they have exactly the same number of nondistinguished real places, where both groups are the compact real form $\mathrm{Sp}(d)$, so the corresponding local classes agree there as well. The Albert--Brauer--Hasse--Noether reciprocity law therefore forces 
\[
\delta_{\mathbb R}(\alpha_v) = \delta_{\mathbb R}(\beta_{v'})\in \operatorname{Br}(\mathbb{R})\cong \mathbb{Z}/2\mathbb{Z}
\]
at the distinguished real place. However, the split form $\mathrm{Sp}(2d,\mathbb R)$ has trivial class in $\operatorname{Br}(\mathbb{R})$, whereas every noncompact quaternionic form $\mathrm{Sp}(m,n)$ has the nontrivial class. This is a contradiction.

For the spin case, following \cite[Proposition 14]{echtler2024bounded}, realize $\mathbf G$ and $\mathbf G'$ by central simple algebras $A/k$ and $A'/l$ equipped with involutions of the first kind and orthogonal type. At corresponding finite places, the local isomorphisms of the algebraic groups induce isomorphisms of the associated Lie algebras. By \cite[Chapter X, Theorem 12]{jacobson2013lie}, an isomorphism between the Lie algebra of skew elements associated with central simple algebras with involution extends to an isomorphism of the underlying algebras that identifies the involutions. Hence, at corresponding finite places, their local Brauer classes agree. The same is true at every nondistinguished real place, where both groups are compact. Hence, the Albert--Brauer--Hasse--Noether reciprocity law gives
\[
[A_v]=[A'_{v'}]\in\operatorname{Br}(\mathbb R)\cong \mathbb{Z}/2\mathbb{Z}
\]
at the distinguished real place. However, for a real orthogonal form $\mathrm{Spin}(m,n)$ the underlying central simple algebra is split, so $[A_v]=0$, whereas $\mathrm{Spin}^*(2d)$ corresponds to the quaternionic algebra $M_d(\mathbb H)$ and therefore has the nontrivial class. This is again a contradiction.
\end{proof}

\begin{proof}[Proof of Theorem~\ref{thm:rigidity}]
    The existence statements in (1) are supplied by Proposition~\ref{prop:classic}, Proposition~\ref{prop:slh}, and Proposition~\ref{prop:exceptional}, which in particular cover the only $E_8$ pair. The conditional $E_6$ case is exactly (2).

    For (3), suppose that lattices \(\Gamma<G\) and \(\Gamma'<G'\) satisfy \(\widehat\Gamma\cong\widehat{\Gamma'}\). No such pair occurs in types \(F_4\) or \(G_2\). Indeed, type \(G_2\) has a unique noncompact real form, namely \(G_{2(2)}\). Type \(F_4\) has two noncompact real forms, \(F_{4(-20)}\) and \(F_{4(4)}\), but the former has real rank \(1\). Thus neither type admits two non-isomorphic real forms both of real rank at least \(2\).
    
    By Proposition~\ref{inner_twist}, $G$ and $G'$ must either both be inner forms or both be outer forms. This excludes the inner--outer pairs of types $A_\ell$, $D_\ell$, and $E_6$. The relevant classical pairs are listed in Table \ref{tab:classical-real-forms}, while the excluded $E_6$ pairs are precisely those with one member in $\{E_{6(-26)},E_{6(6)}\}$ and the other in $\{E_{6(-14)},E_{6(2)}\}$.

    Applying restriction of scalars to $\mathbb Q$, Theorem~\ref{thm:adelic} and Proposition~\ref{prop:symm_space} imply that the dimensions of the associated symmetric spaces are congruent modulo $4$. For the classical real forms, the dimensions used below are
    recorded in Table~\ref{tab:classical-real-forms}. This excludes the following remaining cases:
    \begin{itemize}
        \item For $G\cong\mathrm{Spin}(m,n)$ and $G'\cong\mathrm{Spin}(m',n')$ with $d = m+n=m'+n'$ odd, both groups have absolute Dynkin type \(B_{(d-1)/2}\). We exclude the case in which one of $m,n$ is $0\pmod4$ and one of $m',n'$ is $2\pmod4$, because
        \[
        \dim X=mn\equiv0\pmod4,
        \qquad
        \dim X'=m'n'\equiv2\pmod4.
        \]
        \item For $G\cong\mathrm{SU}(m,n)$ and $G'\cong\mathrm{SU}(m',n')$ with $d = m+n=m'+n'$, both groups have absolute Dynkin type \(A_{d-1}\). We exclude the case in which $m,n$ are both even and $m',n'$ are both odd, because
        \[
        \dim X=2mn\equiv0\pmod4,
        \qquad
        \dim X'=2m'n'\equiv2\pmod4.
        \]
        \item For $G\cong\mathrm{SL}(2n,\mathbb R)$ and $G'\cong\mathrm{SL}(n,\mathbb H)$, both groups have absolute Dynkin type \(A_{2n-1}\). We exclude the case in which $n$ is odd, because
        \[
        \dim X=(2n-1)(n+1),
        \qquad
        \dim X'=(2n+1)(n-1),
        \]
        and $\dim X-\dim X'=2n\equiv2\pmod4$.
        \item For the exceptional type \(E_7\), we exclude the pairs involving $E_{7(-5)}$ and either $E_{7(-25)}$ or $E_{7(7)}$, since their symmetric-space dimensions are respectively $64$, $54$, and $70$.
    \end{itemize}
    After these exclusions, the only remaining pairs are exactly the two families treated in Lemma~\ref{lem:rigidity_brauer}: the pair $\operatorname{Sp}(m,n)$ and $\operatorname{Sp}(2d, \mathbb R)$ of type $C_d$, and the pair $\operatorname{Spin}(m,n)$ and $\operatorname{Spin}^*(d)$ of type $D_{d/2}$. The constructions for (1), (2) in the preceding subsections along with these exclusions for (3) exhaust all pairs of real forms.
\end{proof}
\section{Profinitely Isomorphic Lattices in Semisimple Lie Groups}\label{section4}
Extending the classical constructions in Section~\ref{section3}, we now allow several noncompact archimedean factors, including split real and complex factors. The resulting lattices are irreducible in the ambient products. We also discuss when these constructions can be made explicit. Both the existence results and the explicit constructions are needed for the proof of Theorem~\ref{thm:main}.

\subsection{Constructions in Semisimple Lie Groups}
The following theorem establishes the existence of profinitely isomorphic lattices in pairs of semisimple Lie groups whose factors include different noncompact spin, unitary, or quaternionic symplectic groups, together with possible real split factors.
\begin{theorem}\label{thm:ss_classical}
    Let $r_-,r_1,r_2,r_s\in\mathbb Z_{\geq0}$ with $r_1+r_2 = r_-$ and $r_2\geq1$, and let $m_1,n_1,m_2,n_2\geq1$ satisfy $m_1+n_1=m_2+n_2=d$. In each of the following cases, let $G$ and $G'$ be the displayed Lie groups and assume that $\operatorname{rk}_{\mathbb R}G, \operatorname{rk}_{\mathbb R}G'\geq2$:
    \begin{enumerate}
        \item The quadratic form has dimension $d\geq 5$, 
        \[
        G = \mathrm{Spin}(m_1,n_1)^{r_-},\quad G' = \mathrm{Spin}(m_1,n_1)^{r_1}\times \mathrm{Spin}(m_2,n_2)^{r_2},
        \]
        and one of the following holds:
        \begin{itemize}
            \item $4\mid (m_1-m_2)$ or $4\mid(m_1-n_2)$;
            \item $2\mid (m_1-m_2)$ or $2\mid(m_1-n_2)$, and $2\mid r_2$.
        \end{itemize}
        \item The Hermitian form has dimension $d\geq 3$,
        \[
        G = \mathrm{SU}(m_1,n_1)^{r_-}\times \mathrm{SL}(d,\mathbb{R})^{r_s},\quad G' = \mathrm{SU}(m_1,n_1)^{r_1}\times \mathrm{SU}(m_2,n_2)^{r_2}\times \mathrm{SL}(d,\mathbb{R})^{r_s},
        \]
        and either $2\mid (m_1-m_2)$ or $2\mid(m_1-n_2)$. 
        \item The Hermitian form has dimension $d\geq 2$, and
        \[
        G = \mathrm{Sp}(m_1,n_1)^{r_-}\times \mathrm{Sp}(2d,\mathbb{R})^{r_s},\quad G' = \mathrm{Sp}(m_1,n_1)^{r_1}\times \mathrm{Sp}(m_2,n_2)^{r_2}\times \mathrm{Sp}(2d,\mathbb{R})^{r_s}.
        \]
    \end{enumerate}
    Then there exist irreducible torsion-free cocompact arithmetic lattices $\Gamma<G$ and $\Gamma'<G'$ with $\widehat\Gamma\cong\widehat{\Gamma'}$.
\end{theorem}

\begin{remark}
    After our first draft, Py and Isenrich pointed us to Aka's work. In \cite[Theorem 1]{aka2012profinite}, Aka also constructed pairs of simply-connected absolutely almost simple algebraic groups of the same classical absolute Dynkin type ($D_{(n+1)/2}$ if $n$ is odd and $B_{n/2}$ if $n$ is even) in $\mathrm{Spin}(1,n)$ and $\mathrm{Spin}(5,n-4)$ for $n\ge6$. Note that the specific setting in \cite{aka2012profinite} differs from ours (from the number field to the ambient groups), and we can't see how Aka's construction applies to deduce a result similar to Theorem \ref{thm:main}. Also, after our first draft, Holger Kammeyer and Adrian Baumann pointed us to their recent work \cite{baumann2026congruence}, where they classified ``congruence rigid'' algebraic groups. Their result is related to Theorem \ref{Mainthem} but with different techniques. Moreover, Theorem \ref{Mainthem} with extended results in Section \ref{section4} focus on explicit constructions.
\end{remark}
The finite-place calculations in Lemma~\ref{lem:classic} apply to the semisimple case with only an additional $2\mid(m_1-m_2)$ and even $r_2$ treatment in the spin case:
\begin{lemma}\label{lem:ss_local}
    Let both $\alpha$ and $\alpha - 1$ be algebraic units, such that $\alpha$ is real at every archimedean place, put $k=\mathbb{Q}(\alpha)$, and assume $m_1>m_2$.
    \begin{itemize}
        \item Suppose that $k$ has $r_+$ real places at which $\alpha>1$, $r_1$ real places at which $0<\alpha<1$, and $r_2$ real places at which $\alpha<0$, thus $[k:\mathbb{Q}]=r_++r_1+r_2$. Suppose that $r_+\geq 1$ and that $r_1,r_2,m_1,n_1,m_2,n_2$ and $d$ satisfy the hypotheses of Theorem~\ref{thm:ss_classical}(1), with the divisibility conditions restricted to those involving $(m_1-m_2)$. In the second alternative of that theorem, assume in addition that the ideal $(2)\subset\mathcal O_k$ is prime. Define
        \[
        q=\langle1^{[m_1]},(\alpha-1)^{[n_1]}\rangle,
        \qquad
        q'=\langle1^{[m_2]},\alpha^{[m_1-m_2]},(\alpha-1)^{[n_1]}\rangle.
        \]
        Let $\mathbf G=\mathrm{Spin}(q)$ and $\mathbf G'=\mathrm{Spin}(q')$ be the associated spin $k$-groups. Then
        \[
        \mathbf G\times_k k_v\cong\mathbf G'\times_k k_v
        \]
        for every finite place $v$ of $k$.
        \item Let $a\in \mathcal{O}_k^\times$, and let $E=k(\sqrt a)$ be the corresponding quadratic field extension. Suppose that there are $r_s$ real places with $\sigma_v(a)>0$; among the other real places where $\sigma_v(a)<0$, there are $r_+$ places at which $\alpha>1$, $r_1$ places at which $0<\alpha<1$, and $r_2$ places at which $\alpha<0$. Assume $r_+\geq1$ and that the remaining parameters satisfy Theorem~\ref{thm:ss_classical}(2), with the divisibility conditions restricted to those involving $(m_1-m_2)$. Define the $E/k$-Hermitian forms
        \[
        h=\langle1^{[m_1]},(\alpha-1)^{[n_1]}\rangle,
        \qquad
        h'=\langle1^{[m_2]},\alpha^{[m_1-m_2]},(\alpha-1)^{[n_1]}\rangle,
        \]
        and let $\mathbf G=\mathrm{SU}(h,E/k)$ and $\mathbf G'=\mathrm{SU}(h',E/k)$. Then
        \[
        \mathbf G\times_k k_v\cong\mathbf G'\times_k k_v
        \]
        for every finite place $v$ of $k$.
        \item Let $a,b\in \mathcal{O}_k^\times$, and $B=(a,b)_k$ be the corresponding quaternion algebra. Suppose that there are $r_s$ real places at which $B$ splits; among the other real places where $B$ is ramified, there are $r_+$ places at which $\alpha>1$, $r_1$ places at which $0<\alpha<1$, and $r_2$ places at which $\alpha<0$. Assume $r_+\geq1$ and that the remaining parameters satisfy Theorem~\ref{thm:ss_classical}(3). Define Hermitian forms with respect to the canonical involution on $B$ by
        \[
        h=\langle1^{[m_1]},(\alpha-1)^{[n_1]}\rangle,
        \qquad
        h'=\langle1^{[m_2]},\alpha^{[m_1-m_2]},(\alpha-1)^{[n_1]}\rangle,
        \]
        and let $\mathbf G=\mathrm{Sp}(h,B/k)$ and $\mathbf G'=\mathrm{Sp}(h',B/k)$. Then
        \[
        \mathbf G\times_k k_v\cong\mathbf G'\times_k k_v
        \]
        for every finite place $v$ of $k$.
    \end{itemize}
\end{lemma}
\begin{proof}
    The proof is the same as that of Lemma~\ref{lem:classic}: the common $(\alpha-1)$-block does not affect the determinant comparison, and in the spin case, the ratio of the Hasse invariants is
    \[
    \frac{\epsilon_v(q')}{\epsilon_v(q)} = (\alpha,\alpha)_v^{(m_1-m_2)(m_1-m_2-1)/2}(\alpha,\alpha-1)_v^{(m_1-m_2)n_1} = (\alpha,\alpha)_v^{(m_1-m_2)(m_1-m_2-1)/2}.
    \]
    In the first alternative of the spin case we deduce that $\epsilon_v(q') = \epsilon_v(q)$; in the second alternative, i.e., when $2\mid(m_1-m_2)$ and $2\mid r_2$, it remains to show that $(\alpha,\alpha)_v = 1$ at every finite place. If $v$ lies above an odd rational prime, then $\alpha^{-1}\in\mathcal O_{k_v}^\times$ is a sum of two squares in $k_v$; equivalently,
    \[
    (\alpha,\alpha)_v=1.
    \]
    At a real place, $(\alpha,\alpha)_v=-1$ precisely when $\alpha$ is negative.
    There are $r_2$ such places, which is an even number; thus, the product of the archimedean symbols is $1$. Finally, because $(2)$ is prime in $\mathcal O_k$, there is a unique place $v_2$ above $2$. Hilbert reciprocity therefore gives
    \[
    (\alpha,\alpha)_{v_2}=1.
    \]
    Thus $(\alpha,\alpha)_v=1$ at every finite place, so $\epsilon_v(q)=\epsilon_v(q')$, and the $k_v$-isomorphism follows.
\end{proof}
\begin{proposition}\label{prop:ss_global}
    In the setting of Lemma~\ref{lem:ss_local}:
    \begin{itemize}
        \item In the spin case, there are natural homomorphisms
        \[
        \mathbf G(k)\longrightarrow \mathrm{Spin}(m_1,n_1)^{r_-},
        \qquad
        \mathbf G'(k)\longrightarrow \mathrm{Spin}(m_1,n_1)^{r_1}\times\mathrm{Spin}(m_2,n_2)^{r_2},
        \]
        obtained from the diagonal archimedean embedding followed by projection away from the compact factors. Moreover, there exist {irreducible} torsion-free cocompact arithmetic lattices in these two Lie groups whose profinite completions are isomorphic.
        \item In the unitary settings, there are natural homomorphisms
        \[
        \mathbf G(k)\longrightarrow
        \mathrm{SU}(m_1,n_1)^{r_-}\times\mathrm{SL}(d,\mathbb R)^{r_s},
        \]
        \[
        \mathbf G'(k)\longrightarrow
        \mathrm{SU}(m_1,n_1)^{r_1}\times\mathrm{SU}(m_2,n_2)^{r_2}\times\mathrm{SL}(d,\mathbb R)^{r_s},
        \]
        and irreducible torsion-free cocompact arithmetic lattices in these two Lie groups whose profinite completions are isomorphic.
        \item In the symplectic settings, there are natural homomorphisms
        \[
        \mathbf G(k)\longrightarrow
        \mathrm{Sp}(m_1,n_1)^{r_-}\times\mathrm{Sp}(2d,\mathbb R)^{r_s},
        \]
        \[
        \mathbf G'(k)\longrightarrow
        \mathrm{Sp}(m_1,n_1)^{r_1}\times\mathrm{Sp}(m_2,n_2)^{r_2}\times\mathrm{Sp}(2d,\mathbb R)^{r_s},
        \]
        and irreducible torsion-free cocompact arithmetic lattices in these two Lie groups whose profinite completions are isomorphic.
    \end{itemize}
\end{proposition}
\begin{proof}
    The groups $\mathbf G$ and $\mathbf G'$ are simply connected and absolutely almost simple. Since $r_+\geq1$, the field $k$ is not totally imaginary, and the rank assumptions give archimedean rank at least $2$ for both groups. The same congruence subgroup theorems used in the proof of Proposition~\ref{prop:classic} apply. Together with Lemma~\ref{lem:ss_local}, Lemma~\ref{lem:CSP} therefore yields torsion-free arithmetic subgroups with isomorphic profinite completions.

    It remains to identify their noncompact archimedean factors. There are potentially archimedean places where $\mathbf{G}$ splits or is complex, and potentially more archimedean places with $\sigma(\alpha)<0$:

    (1) At a real place $v$ with $\sigma_v(\alpha)>1$, the forms $q,q'$ are positive definite and the corresponding factor is $\mathrm{Spin}(d)$. At a real place with $0<\sigma_v(\alpha)<1$, the factor for both is $\mathrm{Spin}(m_1,n_1)$. At a real place with $\sigma_v(\alpha)<0$, the factor for $\mathbf{G}$ is still $\mathrm{Spin}(m_1,n_1)$, while which for $\mathbf{G'}$ is $\mathrm{Spin}(m_2,n_2)$. Hence
    \[
    \mathbf G(k\otimes_\mathbb Q\mathbb R)
    \cong
    \mathrm{Spin}(d)^{r_+}\times
    \mathrm{Spin}(m_1,n_1)^{r_-},
    \]
    and
    \[
    \mathbf G'(k\otimes_\mathbb Q\mathbb R)
    \cong
    \mathrm{Spin}(d)^{r_+}\times
    \mathrm{Spin}(m_1,n_1)^{r_1}\times
    \mathrm{Spin}(m_2,n_2)^{r_2}.
    \]

    (2) At a real place with $\sigma_v(a)>0$, the quadratic algebra $E_v$ is split and the factor is $\mathrm{SL}(d,\mathbb R)$. At a real place with $\sigma_v(a)<0$, one has $E_v\cong\mathbb C$; the factor is $\mathrm{SU}(d)$, $\mathrm{SU}(m_1,n_1)$, or $\mathrm{SU}(m_2,n_2)$ for the corresponding cases similar to (1). Consequently,
    \[
    \mathbf G(k\otimes_\mathbb Q\mathbb R)
    \cong
    \mathrm{SU}(d)^{r_+}\times\mathrm{SU}(m_1,n_1)^{r_-}\times
    \mathrm{SL}(d,\mathbb R)^{r_s},
    \]
    and $\mathbf G'(k\otimes_\mathbb Q\mathbb R)$ is isomorphic to a mixed product in a similar way.

    (3) At a real place where $B$ splits, Morita equivalence gives the split factor $\mathrm{Sp}(2d,\mathbb R)$. At a real place where $B$ is ramified, the factor is $\mathrm{Sp}(d)$, $\mathrm{Sp}(m_1,n_1)$, or $\mathrm{Sp}(m_2,n_2)$ for the corresponding cases. Thus
    \[
    \mathbf G(k\otimes_\mathbb Q\mathbb R)
    \cong
    \mathrm{Sp}(d)^{r_+}\times\mathrm{Sp}(m,n)^{r_-}\times
    \mathrm{Sp}(2d,\mathbb R)^{r_s},
    \]
    and $\mathbf G'(k\otimes_\mathbb Q\mathbb R)$ is isomorphic to a mixed product in a similar way. Since $r_+\geq1$, both groups have a compact archimedean factor. Corollary~\ref{cor:CSP} therefore gives the required irreducible torsion-free cocompact lattices in the displayed noncompact factors.
\end{proof}
To complete the proof of Theorem~\ref{thm:ss_classical}, it remains to choose the number field with the expected archimedean behavior:
\begin{proof}[Proof of Theorem~\ref{thm:ss_classical}]
    Since $\mathrm{Spin}(m_2,n_2)\cong\mathrm{Spin}(n_2,m_2)$, we may interchange $m_2$ and $n_2$ and assume that $4\mid(m_1-m_2)$ in the first alternative of Theorem~\ref{thm:ss_classical}(1), or $2\mid(m_1-m_2)$ in the second. Similarly, since $\mathrm{SU}(m_2,n_2)\cong\mathrm{SU}(n_2,m_2)$, we assume $2\mid (m_1-m_2)$ in Theorem~\ref{thm:ss_classical}(2). Up to exchanging $m_1$ with $n_1$ and $m_2$ with $n_2$, we assume $m_1>m_2$. In all three cases, we fix an auxiliary integer $r_+\geq 2$.

    (1) We claim the existence of an algebraic unit $\alpha$ such that $\alpha-1\in \mathcal{O}_{\mathbb{Q}(\alpha)}^\times$ and with predicted numbers of real places in the three intervals. Let
    \[
    D=r_-+r_+ = r_1+r_2 + r_+.
    \]
    Choose a polynomial $g(x)\in\mathbb{Z}[x]$ of degree $D-3$ with positive leading term denoted by $c$, $r_+-2$ distinct real roots $x_i>1$, $r_1$ distinct roots $0<x_i<1$, and $r_2-1$ distinct negative roots. Choose an irreducible polynomial $\overline f(x)\in\mathbb F_2[x]$ of degree $D$, and a monic lift $f_0(x)\in\mathbb Z[x]$ with $f_0(1)=(-1)^{r_+}$ and $f_0(0) = (-1)^{r_++r_1}$. Set
    \[
    f_N(x)=f_0(x)-2Nx(x-1)g(x).
    \]
    Then $(2N)^{-1}f_N(x)\to -x(x-1)g(x)$ coefficientwise away from the leading term as $N\to\infty$. For all sufficiently large $N$, the polynomial $f_N$ has $r_+-2$ roots $\alpha_i>1$, $r_1$ roots $0<\alpha_i<1$, $r_2-1$ negative roots, and additional real roots $\alpha_0$, $\alpha_1$ and $\alpha_\infty$ sufficiently close to $0$, $1$, and with sufficiently large absolute value, respectively; all roots are distinct. Since $f_N$ is monic, the sum of the roots is $2Nc+O(1)\to\infty$, so $\alpha_\infty>0$; the product of the roots is $(-1)^D f_N(0)=(-1)^{r_2}$, and $\prod(\alpha_i-1) = (-1)^D f_N(1)=(-1)^{r_-}$, so $\alpha_1 > 1$ and $\alpha_0 < 0$. Hence $f_N$ has exactly $r_+$ roots $\alpha_i>1$, $r_1$ roots $0<\alpha_i<1$, and $r_2$ negative roots.

    The reduction of $f_N$ modulo $2$ is $\overline f$, so $f_N$ is irreducible over $\mathbb Q$. Let $\alpha$ be a root and put $k=\mathbb Q(\alpha)$. The embeddings $k\hookrightarrow\mathbb C$ are obtained by sending $\alpha$ to the roots of $f_N$, so $k$ has precisely the prescribed archimedean sign pattern. Since $\overline f$ is irreducible and separable, $(2)$ is prime in $\mathcal O_k$. Proposition~\ref{prop:ss_global}, case (1) now gives the required lattices.

    (2) We claim the existence of an algebraic unit $\alpha$, and $a\in \mathcal{O}_{\mathbb{Q}(\alpha)}^\times$, such that $\alpha-1\in \mathcal{O}_{\mathbb{Q}(\alpha)}^\times$ with predicted numbers of split places and non-split places in the three intervals. If $r_s = 0$, we construct $f = f_N$ same as above, and take $a = -1$. If $r_s\geq 1$, we modify the construction by setting
    \[
    f_N(x) = f_0(x) - 2Nx(x-1)(x+1)g(x),
    \]
    where the monic lift $f_0\in\mathbb{Z}[x]$ satisfies $f_0(1) = (-1)^{r_+}$, $f_0(0) = (-1)^{r_++r_1}$ and $f_0(-1) = (-1)^{r_++r_-}$, and the polynomial $g\in\mathbb{Z}[x]$ of degree $D-4$ with positive leading term has $r_+-2$ distinct roots $x_i>1$, $r_1$ distinct roots $0<x_i<1$, $r_2-1$ distinct roots $-1<x_i<0$, and $r_s - 1$ distinct roots $x_i<-1$. As $N\to\infty$, a similar Vieta argument guarantees that $f = f_N$ has exactly $r_+$, $r_1$, $r_2$, and $r_s$ roots in these four intervals. Choosing $a = -\alpha-1$, the condition $|f(-1)| = 1$ implies that $a$ is a unit, and the number of roots of $f$ in each interval gives the prescribed $r_s$ split real places of $E=k(\sqrt a)$, and the $r_+$, $r_1$ and $r_2$ ramified places in each cases. Proposition~\ref{prop:ss_global}, case (2) gives the required lattices.

    (3) Choose $\alpha$ with the required sign pattern as above, and take $a=b = -1$ or $a = b = -\alpha-1$ corresponding to case (2), so the quaternion algebra $B=(a,b)_k$ splits at the prescribed $r_s$ real places and is ramified at the prescribed $r_+$, $r_1$ and $r_2$ real places. Proposition~\ref{prop:ss_global}, case (3) then gives the required lattices.
\end{proof}
\begin{remark}
    We can allow complex factors $\operatorname{Spin}(d,\mathbb{C})$, $\operatorname{SU}(d,\mathbb{C})$, and $\operatorname{Sp}(2d,\mathbb{C})$ in the three cases of Theorem~\ref{thm:ss_classical}, respectively, by allowing the field $\mathbb{Q}(\alpha)$ to have complex places. This is achieved by including pairs of nonreal conjugate roots in the auxiliary polynomial $g(x)$.

    In the unitary case, the parity restrictions on the signatures can also be replace by the condition of $2\mid r_2$. This is achieved by a Hilbert reciprocity argument analogous to the spin case.
\end{remark}
\subsection{Explicit Construction of Profinitely Isomorphic Lattices}\label{subsec:explicit}
For certain cases of the semisimple classical (Theorem \ref{thm:ss_classical}) and exceptional (Proposition \ref{prop:exceptional} constructions, there exists level-$3$ congruence subgroups $\Gamma<\mathbf{G}$, $\Gamma'<\mathbf{G}'$ with the prescribed profinite isomorphism. The reason is the triviality of the congruence kernels and compatible integral structures at all finite places.
\begin{lemma}\label{lem:trivial_kernel}
    Let $\mathbf{G}$ be a connected, simply connected, absolutely almost simple algebraic group over a number field $k$. Assume that $\mathbf{G}$ satisfies the congruence subgroup property, and at some real place $v$, the Lie group $\mathbf{G}(k_v)$ is isomorphic to one of the following:
    \begin{itemize}
        \item $\operatorname{Spin}(m,n)$ with $m+n\geq 5$ and $m,n\geq 2$.
        \item $\operatorname{SU}(m,n)$ with $m+n\geq 3$, or $\operatorname{SL}(d,\mathbb{R})$ with $d\geq 3$.
        \item $\operatorname{Sp}(2d,\mathbb{R})$, $d\geq 2$.
        \item Exceptional real forms $E_{7(-25)}$, $E_{7(7)}$, $E_{8(-24)}$, or $E_{8(8)}$.
    \end{itemize}
    Then the congruence kernel of every arithmetic lattice $\Gamma<\mathbf{G}(k)$ is trivial.
\end{lemma}
\begin{proof}
    Since $\mathbf{G}$ is assumed to satisfy the CSP, \cite[Theorem 2]{prasad2010developments} then implies that this congruence kernel is isomorphic to the Pontrjagin dual of the metaplectic kernel. At the distinguished real place $v$, $\mathbf{G}(k_v)$ is isomorphic to one of the real forms listed above, which is not topologically simply connected. Therefore, \cite[Theorem 3]{prasad2010developments} implies that the metaplectic kernel of $\mathbf{G}$ is trivial, and the congruence kernel of $\mathbf{G}$ is also trivial. Consequently, the arithmetic and congruence topologies coincide on $\mathbf{G}(k)$, and hence on every arithmetic subgroup of $\mathbf{G}(k)$.
\end{proof}
Consider the following pairs of Lie groups $G,G'$:
\begin{itemize}
    \item $G,G'$ are Lie groups in Theorem \ref{thm:ss_classical}, case (1), with $m_1,n_1\geq 2$, either $m_2,n_2\geq 2$ or $r_1>0$, and either $4\mid(m_1-m_2)$ or $2\mid(m_1-m_2)$ and $2\mid r_2$. This includes the corresponding case in Proposition \ref{prop:classic}.
    \item $G,G'$ are Lie groups in Theorem \ref{thm:ss_classical}, case (2), with $2\mid(m_1-m_2)$. This includes the corresponding case in Proposition \ref{prop:classic}.
    \item $G,G'$ are Lie groups in Theorem \ref{thm:ss_classical}, case (3), with $r_s\geq 1$.
    \item $G,G'$ are Lie groups in Proposition \ref{prop:exceptional} corresponding to the Lie algebras $E_7$ or $E_8$.
\end{itemize}
Let $\mathbf{G}$ and $\mathbf{G}'$ be the algebraic $k$-groups constructed in the corresponding proofs. For each $\mathbf{G}$, we choose a faithful $k$-representation $\rho:\mathbf{G}\to\operatorname{GL}_k(W)$ and a full $\mathcal{O}_k$-lattice $\Lambda\subset W$ as follows.
\begin{itemize}
    \item For $\mathbf{G} = \operatorname{Spin}(k,q)$, take $W=\operatorname{Cl}^+(k^d,q)$ and let $\rho$ be the representation by left multiplication. Set
    \[
    \Lambda=\mathscr{C}^+(k^d,q) =\bigoplus_{\substack{0\leq 2s\leq d\\1\leq i_1<\cdots<i_{2s}\leq d}}\mathcal{O}_k e_{i_1}\cdots e_{i_{2s}}.
    \]
    \item For $\mathbf{G} = \operatorname{SU}(E/k,h)$, we take the matrix representation on $W = E^d$, seen as a $k$-vector space, and $\Lambda = \mathcal{O}_E^d$.
    \item For $\mathbf{G} = \operatorname{Sp}(B/k,h)$, we take the matrix representation on $W = B^d$, seen as a $k$-vector space, and $\Lambda = \mathcal{O}_B^d$, where $\mathcal{O}_B$ is a fixed maximal $\mathcal{O}_k$-order in $B$.
    \item For $\mathbf{G}=G(H_3(C),B)$, use its realization as the isometry group of the generalized Freudenthal triple system on $\operatorname{Mat}_{28}(B)$. Take the natural matrix representation on $W = B^{28}$, viewed as a $112$-dimensional $k$-vector space.
    \item For $\mathbf{G}=G(H_3(C),C)$, take $W=\operatorname{Lie}(\mathbf{G})$ and $\rho=\operatorname{Ad}$, so that $\dim_k W=248$.
\end{itemize}
We define
\[
\Gamma= \{g\in\mathbf{G}(k):\rho(g)\Lambda=\Lambda,\ (\rho(g)-\operatorname{id}_W)\Lambda\subseteq3\Lambda\},
\]
and define $\Gamma'\leq\mathbf{G}'(k)$ similarly. We identify these subgroups with their images in $G$ and $G'$ under the archimedean projections. In the exceptional cases, the full $\mathcal{O}_k$-lattices $\Lambda\subset W$ and $\Lambda'\subset W'$ are chosen so that, at every finite place $v$, there exist a $k_v$-isomorphism $\varphi_v:\mathbf{G}_{k_v}\to\mathbf{G}'_{k_v}$ and a linear isomorphism $f_v:W_v\to W'_v$ satisfying
\begin{equation}\label{equ:compatible}
    f_v\rho(g)=\rho'(\varphi_v(g))f_v, \quad f_v\Lambda_v=\Lambda'_v.
\end{equation}
Such choices will be constructed in the proof below.

    
\begin{lemma}\label{lem:completion_is_prod}
    Under these settings, we have
    \[
    \prod_{v\in V_f(k)}\Gamma_v \cong \widehat{\Gamma}.
    \]
\end{lemma}
\begin{proof}
    Since the algebraic group $\mathbf{G}$ is connected, simply connected, absolutely almost simple and has at least one noncompact archimedean place, strong approximation implies that the image of the diagonal embeddings $\mathbf{G}(k)\hookrightarrow \mathbf{G}(\mathbb{A}_k^f)$ is dense \cite[Theorem 7.12]{platonov1994algebraic}. Moreover, the relations
    \[
    W\cap\prod_{v\in V_f(k)}\Lambda_v=\Lambda,
    \quad
    W\cap\prod_{v\in V_f(k)}3\Lambda_v=3\Lambda,
    \]
    together with the level-$3$ definition of $\Gamma$, imply the identification
    \[
    \Gamma = \mathbf{G}(k)\cap \prod_{v\in V_f(k)}\Gamma_v.
    \]
    The denseness of $\mathbf{G}(k)$ thus implies that $\Gamma$ is dense in $\prod_{v\in V_f(k)}\Gamma_v$, so the congruence completion
    \[
    \overline{\Gamma} = \prod_{v\in V_f(k)}\Gamma_v.
    \]
    By Lemma \ref{lem:trivial_kernel}, both $\Gamma$ and $\Gamma'$ have trivial congruence kernels. Therefore, 
    \[
    \widehat{\Gamma}\cong \overline{\Gamma}\cong \prod_{v\in V_f(k)}\Gamma_v,
    \]
\end{proof}
\begin{theorem}\label{prop:lattice_level_3}
    Suppose that the minimal polynomial of $\alpha$ over $\mathbb{Q}$ is irreducible modulo $2$. Then $\Gamma<G$ and $\Gamma'<G'$ are irreducible, torsion-free, cocompact arithmetic lattices with isomorphic profinite completions.
\end{theorem}
\begin{proof}
    It is clear that level-$3$ congruence subgroups are torsion-free, and Corollary \ref{cor:CSP} guarantees the irreducibility and cocompactness. 
    
    For the profinitely isomorphism, Lemma \ref{lem:completion_is_prod} identifies $\widehat{\Gamma} = \prod_{v\in V_f(k)}\Gamma_v$ and $\widehat{\Gamma}' = \prod_{v\in V_f(k)}\Gamma'_v$. Thus, it suffices to identify an appropriate isomorphism between the adelic groups,
    \[
    \Phi: \mathbf{G}(\mathbb{A}_k^f)\overset{\sim}{\to}\mathbf{G}'(\mathbb{A}_k^f),
    \]
    that is restricted to an isomorphism between $\prod_{v\in V_f(k)}\Gamma_v$ and $\prod_{v\in V_f(k)}\Gamma'_v$.

    \textbf{Case (1): Classical groups.} For the spin, unitary, and symplectic cases, our construction assumed $m_1>m_2$ without loss of generality. In the spin case, we claim that at every finite place $v$, there exists numbers $x_v,y_v,z_v,w_v\in \mathcal{O}_{k,v}$ such that 
    \[
    x_v^2+y_v^2+z_v^2+w_v^2 = \alpha.
    \]
    Indeed, for a place $v$ above an odd prime, $[\alpha]\in \mathbb{F}_q$ is a sum of two squares, and the Hensel lift gives a two-square expression for $\alpha$. For $v\mid 2$, the irreducibility of the minimal polynomial implies $k_v/\mathbb Q_2$ is unramified. Therefore, \cite[Proposition 4.3]{he2023k} guarantees a four-square expression for $\alpha$.
    
    Now, taking the matrix,
    \[
    R_0 = \begin{pmatrix}
        x_v & -y_v & -z_v & -w_v \\
        y_v & x_v & -w_v & z_v \\
        z_v & w_v & x_v & -y_v \\
        w_v & -z_v & y_v & x_v
    \end{pmatrix}\in \operatorname{Mat}_4(\mathcal{O}_{k,v}).
    \]
    Then $R_0^\mathsf{T}R_0 = \alpha I_4$. If $4\mid (m_1-m_2)$, then for $R = \operatorname{diag}(I_{m_1},{R_0^{-1}}^{[(m_1-m_2)/4]},I_{n_2})$, $R^\mathsf{T}Q'R = Q$, where $Q$, $Q'$ are the matrices define the quadratic forms $q,q'$. 

    If instead $2\mid r_2$ and $2\mid(m_1-m_2)$, we claim the existence of $x_v,y_v\in\mathcal{O}_{k,v}$ such that $x_v^2 + y_v^2 = \alpha$, particularly for $v\mid 2$. In fact, the Hilbert reciprocity argument in the proof of Lemma \ref{lem:ss_local} implies $x_v^2 + y_v^2 = \alpha$ for certain $x_v,y_v\in k_v$. Since $k_v/\mathbb{Q}_2$ is unramified, $v_2(x_v^2+y_v^2)\leq \min\{2v_2(x_v),2v_2(y_v)\}+1$, which forces $x_v$, $y_v$ to be integers. Thus, we can take
    \[
    R_0 = \begin{pmatrix}
        x_v & -y_v \\
        y_v & x_v
    \end{pmatrix},
    \]
    and $R = \operatorname{diag}(I_{m_1},{R_0^{-1}}^{[(m_1-m_2)/2]},I_{n_2})$.
    
    In both cases, $x\mapsto Rx$ gives a isomorphism $f_v:(k_v^d,q)\overset{\sim}{\to} (k_v^d,q')$ between quadratic space. Since $x_v,\dots, w_v\in \mathcal{O}_v$, and $\alpha$ is a unit, we have $R\in \operatorname{GL}(d,\mathcal{O}_{k,v})$.
    
    The matrix $R$ further induces the isomorphism between Clifford algebras, $\varphi_v:\operatorname{Cl}^+(k_v^d,q)\overset{\sim}{\to} \operatorname{Cl}^+(k_v^d,q')$, restricted to the algebraic group isomorphism $\mathbf{G}(k_v)\overset{\sim}{\to}\mathbf{G}'(k_v)$. All isomorphisms at finite places together give $\Phi: \mathbf{G}(\mathbb{A}_k^f)\overset{\sim}{\to}\mathbf{G}'(\mathbb{A}_k^f)$. Since $R\in \operatorname{GL}(d,\mathcal{O}_{k,v})$, the induced Clifford algebra isomorphism $\varphi_v$ sends $\mathscr{C}^+(k_v^d,q)$ to $\mathscr{C}^+(k_v^d,q')$. Consequently, $\varphi_v$ induces an isomorphism between the lattices $\Gamma_v$ and $\Gamma'_v$. Therefore, $\Phi$ gives an isomorphism between the compact open subgroups,
    \[
    \prod_{v}\Gamma_v\cong \prod_{v}\Gamma'_v.
    \]
    In the unitary case, we claim there exists $\xi_v,\zeta_v\in \mathcal{O}_{E,v} = \mathcal{O}_E\otimes_{\mathcal{O}_k}\mathcal{O}_{k,v}$, such that
    \[
    N_{E_v/k_v}(\xi_v)+ N_{E_v/k_v}(\zeta_v)= \xi_v\overline{\xi}_v + \zeta_v\overline{\zeta}_v = \alpha.
    \]
    This follows our two-square expression in (1) for each $v\nmid 2$; for $v\mid 2$, since $a\in k$ is assumed to be a unit, the existence of $x_v$, $y_v$, $z_v$ and $w_v\in \mathcal{O}_{k,v}$ with $x_v^2 - ay_v^2 + z_v^2 - aw_v^2 = \alpha$ is again guaranteed by \cite[Proposition 4.3]{he2023k}. This implies that $\xi_v = x_v + y_v\sqrt{a}$ and $\zeta_v = z_v + w_v\sqrt{a}\in\mathcal{O}_{E,v}$ satisfy our claim.

    Now, we take the matrix
    \[
    R_0 = \begin{pmatrix}
        \xi_v & -\overline{\zeta}_v \\ \zeta_v & \overline{\xi}_v
    \end{pmatrix}\in \operatorname{Mat}_2(\mathcal{O}_{E,v}),
    \]
    then $R_0^*R_0 = \alpha I_2$ and $\det(R_0) = \alpha$. Define $R = \operatorname{diag}(I_{m_1},{R_0^{-1}}^{[(m_1-m_2)/2]},I_{n_2})$, then $R\in \operatorname{GL}(d,\mathcal{O}_{E,v})$, and $R^* H'R = H$, where $H$, $H'$ are the matrices defining the Hermitian forms $h,h'$. This gives the algebraic group isomorphism,
    \[
    \varphi_v: \mathbf{G}(k_v)\overset{\sim}{\to}\mathbf{G}'(k_v),\quad X\mapsto RXR^{-1}.
    \]
    The remainer of the proof is similar to the spin case.

    In the symplectic case, since the reduced norm map $\operatorname{Nrd}: \mathcal{O}_{B,v}^\times\to \mathcal{O}_{k,v}^\times$ is surjective, there exists $\xi_v\in\mathcal{O}_{B,v}$ such that $\operatorname{Nrd}(\xi_v) = \alpha$. Define $R = \operatorname{diag}(I_{m_1}, \xi_v^{-1} I_{m_1-m_2}, I_{n_2})$, then $R\in \operatorname{GL}(d,\mathcal{O}_{B,v})$ and $R^*H'R = H$. The remainder of the proof is again similar to the spin and unitary cases.

    \textbf{Case (2): Exceptional groups.} Recall that we defined in the proof of Proposition \ref{prop:exceptional} that
    \[
    C = (a,b,c)_k,\ C' = (a,b,c')_k,\ A = H_3(C),\ A' = H_3(C'),
    \]
    where $a,b,c,c'\in \mathcal{O}_k^\times$. Let $Q = (a,b)_k$, and we fix a maximal $\mathcal{O}_k$-order $\mathcal{O}_Q$. For each finite place $v$, there exists $\xi_v\in \mathcal{O}_{Q,v}$ such that $\operatorname{Nrd}(\xi_v) = c/c'$; since $c,c'$ are units, $\xi_v$ is a unit in $\mathcal{O}_{Q,v}$. Denote $C = Q\oplus Q\ell$ and $C' = Q\oplus Q\ell'$, where $\ell^2 = c$ and $\ell'^2 = c'$; at each finite place $v$, define an isomorphism,
    \[
    \psi_v: C_v\to C'_v,\quad \psi_v(x+y\ell) = x + \xi_vy\ell.
    \]
    It is clear that $\psi_v$ identifies $\mathcal{O}_{C,v}$ with $\mathcal{O}_{C',v}$, where $\mathcal{O}_C\coloneqq \mathcal{O}_Q\oplus \mathcal{O}_Q\ell$ and $\mathcal{O}_C'\coloneqq \mathcal{O}_Q\oplus \mathcal{O}_Q\ell'$. This induces an isomorphism between the corresponding Albert algebras:
    \[
    \psi_v(A) = A',\quad \psi_v(\mathcal{O}_A) = \mathcal{O}_{A'},
    \]
    where $\mathcal{O}_A = H_3(\mathcal{O}_C)$, and $\mathcal{O}_{A'} = H_3(\mathcal{O}_{C'})$. Thus isomorphism between the algebraic groups,
    \[
    \varphi_v:\ G(A_v) = \mathbf{G}(k_v)\overset{\sim}{\to}\mathbf{G}'(k_v) = G(A'_v).
    \]
    For the lattices $\Lambda\subset W$, $\Lambda'\subset W'$, and the isomorphism $f_v$ satisfying \eqref{equ:compatible}, we discuss the $E_7$ and $E_8$ cases separately. In the $E_7$ case, realizing
    \[
    W = (k\oplus A)\otimes_k B,\quad W' = (k\oplus A')\otimes_k B,
    \]
    then the linear isomorphism $f_v:W_v\to W_v'$ can be taken as
    \[
    f_v = (\operatorname{id}_{k_v}\oplus \psi_v)\otimes \operatorname{id}_{B_v}.
    \]
    The isomorphism $f_v$ is compatible with $\varphi_v$, giving the first claim in \eqref{equ:compatible}. Define
    \[
    \Lambda = (\mathcal{O}_k\oplus \mathcal{O}_A)\otimes_{\mathcal{O}_k}\mathcal{O}_B,\quad \Lambda' = (\mathcal{O}_k\oplus \mathcal{O}_{A'})\otimes_{\mathcal{O}_k}\mathcal{O}_B,
    \]
    then the second claim in \eqref{equ:compatible} is guaranteed.

    In the $E_8$ case, realizing the Lie algebras
    \[
    W = \operatorname{Lie}(G(A,C)) = \operatorname{Der}(A)\oplus \operatorname{Der}(C)\oplus(A_0\otimes_k C_0),
    \]
    and
    \[
    W' = \operatorname{Lie}(G(A',C)) = \operatorname{Der}(A')\oplus \operatorname{Der}(C)\oplus(A'_0\otimes_k C_0),
    \]
    where $A_0$ and $C_0$ mean the corresponding trace-free subalgebras. The linear isomorphism,
    \[
    f_v = \operatorname{d}\varphi_v= (\psi_v)_*\oplus \operatorname{id}\oplus(\psi_v\otimes \operatorname{id}),
    \]
    naturally satisfies the first claim in \eqref{equ:compatible}. Take
    \[
    \Lambda = \operatorname{Der}(\mathcal{O}_A)\oplus \operatorname{Der}(\mathcal{O}_C)\oplus (\mathcal{O}_{A_0}\otimes_{\mathcal{O}_k}\mathcal{O}_{C_0}),\ \Lambda' = \operatorname{Der}(\mathcal{O}_{A'})\oplus \operatorname{Der}(\mathcal{O}_C)\oplus (\mathcal{O}_{A'_0}\otimes_{\mathcal{O}_k}\mathcal{O}_{C_0}),
    \]
    then the second claim in \eqref{equ:compatible} also holds.
    
    Now, for $g\in \Gamma$, \eqref{equ:compatible} gives
    \[
    \rho'(\varphi_v(g))\Lambda' = \rho'(\varphi_v(g))f_v\Lambda = f_v\rho(g)\Lambda = f_v\Lambda = \Lambda',
    \]
    and similarly $(\rho'(\varphi_v(g)) - \operatorname{id}_{W'})\Lambda' \subseteq 3\Lambda'$. Converse relations also hold, thus $\varphi_v$ restricts to an isomorphism between $\Gamma_v$ and $\Gamma'_v$. The product isomorphism $\Phi$ thus restricts to an isomorphism between $\prod_v\Gamma_v$ and $\prod_v\Gamma'_v$, completes our proof.
\end{proof}

\section{K\"ahlerness and Profinite Rigidity of Arithmetic Lattices}\label{section5}
Based on the results developed in the previous sections, we construct two families of finitely presented and residually finite groups $G_1$ and $G_2$ so that they have the same profinite completions $\widehat{G}_1\cong\widehat{G}_2$,  moreover $G_1$ is a K\"ahler group and $G_2$ is non-K\"ahler.

\subsection{The K\"ahler lattice}

First of all, we recall the following proposition originally due to Guichardet and Wigner but reformulated by Wienhard:

\begin{proposition}[{\cite[Chapter 3, Proposition 1.2]{wienhard2005bounded},\cite{guichardet1978cohomologie}}] \label{prop:Anna}
Let $X=G/K$ be an irreducible Riemannian symmetric space, where $K$ is a maximal compact subgroup of a connected semisimple, with finite center and no compact factors, real Lie group $G$. Then $X$ is Hermitian if it satisfies one of the following equivalent properties:
\begin{enumerate}
    \item $X$ admits an $\operatorname{Isom}(X)^{\circ}$-invariant Kähler form;
    \item $X$ carries an $\operatorname{Isom}(X)^{\circ}$-invariant complex structure;
    \item The center $\mathfrak{c}(\mathfrak{k})$ of $\mathfrak{k}$ is nontrivial, where $\mathfrak{k}$ is the Lie algebra of $K$.
\end{enumerate}
Here $\operatorname{Isom}(X)^{\circ}$ is the identity component of the isometry group of $X$.
\end{proposition}

\begin{proposition}\label{prop:G1-kahler}
  Let $\Gamma_1$ be a torsion-free cocompact lattice in $\prod_{i=1}^k\operatorname{Spin}(n_i,2)^{r_i}$ for integers $n_i\geq 3, r_i\geq 1$. Then $\Gamma_1$ is a K\"ahler group.
\end{proposition}
\begin{proof}
 For any integer $n_i\geq 3$, denote by $K_i$ a maximal compact subgroup of $\operatorname{Spin}(n_i,2)$ and $X_i\coloneqq \operatorname{Spin}(n_i,2)/K_i$ the irreducible symmetric space of $\operatorname{Spin}(n_i,2)$. {Identifying
 \[
 \prod_{i=1}^k X_i^{r_i}\cong \bigl(\prod_{i=1}^k \operatorname{Spin}(n_i,2)^{r_i}\bigr)/\bigl(\prod_{i=1}^k K_i^{r_i}\bigr),
 \]
 the product $\prod_{i=1}^k \operatorname{Spin}(n_i,2)^{r_i}$ acts on $\prod_{i=1}^k X_i^{r_i}$ properly. Thus, the action of the lattice $\Gamma_1<\prod_{i=1}^k \operatorname{Spin}(n_i,2)^{r_i}$ on $\prod_{i=1}^k X_i^{r_i}$ is properly discontinuous. Moreover, for any $x\in \prod_{i=1}^k X_i^{r_i}$, the stabilizer $\operatorname{Stab}(x)<\prod_{i=1}^k \operatorname{Spin}(n_i,2)^{r_i}$ is compact. The discreteness of $\Gamma_1$ then implies that $\Gamma_1\cap \operatorname{Stab}(x)$ is finite; as $\Gamma_1$ is torsion-free, this intersection is indeed trivial. Thus, the action of $\Gamma_1$ on $\prod_{i=1}^k X_i^{r_i}$ is free.
 }
 
 Next, consider the Cartan decomposition of the Lie algebra $\mathfrak{g}$ of $\operatorname{Spin}(n_i,2)$ or $\mathrm{SO}_0(n_i,2)$ $$\mathfrak{g}=\mathfrak{k}\oplus \mathfrak{p},$$ where $\mathfrak{k}=\mathfrak{so}(n_i)\oplus\mathfrak{so}(2)$ is the Lie algebra of $K_i$. Note that the center of $\mathfrak{k}$ is $$\mathfrak{c}(\mathfrak{k})=\mathfrak{c}(\mathfrak{so}(n_i))\oplus\mathfrak{c}(\mathfrak{so}(2))=\mathfrak{so}(2)\neq 0.$$ since $\mathfrak{so}(n_i)$ has trivial center for $n_i\geq 3$ and $\mathfrak{so}(2)$ is $\mathbb R$.  Applying  Proposition \ref{prop:Anna}, we then have $\operatorname{Spin}(n_i,2)$ is of Hermitian type, and $X_i$ admits an $\operatorname{Isom}(X_i)^{\circ}$-invariant complex structure $J_i$ and K\"ahler form $\omega_i$. {The natural action of $\operatorname{Spin}(n_i,2)$ on $X_i$ factors through the adjoint quotient
    \[
    \rho_i:\operatorname{Spin}(n_i,2)\longrightarrow
    \operatorname{PSO}^+(n_i,2)
    \cong \operatorname{Isom}(X_i)^\circ.
    \]
Hence the complex structure $J_i$ and the K\"ahler form $\omega_i$ are $\operatorname{Spin}(n_i,2)$-invariant. The product of these K\"ahler structures gives the K\"ahler structure on $\prod_{i=1}^k X_i^{r_i}$, which is $\prod_{i=1}^k \operatorname{Spin}(n_i,2)^{r_i}$-invariant. As a lattice, $\Gamma_1$ thus also preserves the complex structure and K\"ahler form on the product space.}

 Since $\Gamma_1$ is a cocompact lattice of  $\prod_{i=1}^k\operatorname{Spin}(n_i,2)^{r_i}$, $\Gamma_1\backslash \prod_{i=1}^k X_i^{r_i}$ is a compact K\"ahler manifold. The free and properly discontinuous action of $\Gamma_1$ on the simply connected space $\prod_{i=1}^k X_i^{r_i}$ implies
 $$\pi_1( \Gamma_1\backslash \prod_{i=1}^k X_i^{r_i})\cong \Gamma_1.$$
 Therefore, $\Gamma_1$ is a K\"ahler group.
\end{proof}

\subsection{The non-K\"ahler lattices} In this subsection, we show that any torsion-free cocompact arithmetic lattice $\Gamma_2$ of products of $\operatorname{Spin}(n,1)$ ($n\geq 3$) is non-K\"ahler. We first give some necessary concepts and tools used in the discussion on the non-K\"ahlerness of $\Gamma_2$. Recall that a \textit{hyperbolic 2-orbifold} is a quotient $\Sigma = \Lambda\backslash\mathbb{H}^2$. Here $\Lambda<\mathrm{PSL}(2,\mathbb{R})\cong \operatorname{Isom}(\mathbb{H}^2)^\circ$ is discrete and cocompact; in this case, $\Lambda\cong\pi_1^{\mathrm{orb}}(\Sigma)$. We have the following important theorem due to Delzant-Py.
\begin{theorem}[{\cite[Theorem 1, Corollary 1]{delzant2012kahler}}]\label{thm:DP}

Let $\Gamma=\pi_1(M)$ be a K\"ahler group corresponding to a compact K\"ahler manifold $M$. If for some $n\geq 3$, a homomorphism $\rho: \Gamma\rightarrow\operatorname{Isom}(\mathbb H^n)$ has Zariski-dense image, then there exist a hyperbolic $2$-orbifold $\Sigma$, such that $\rho$ factors through the map $f_*:\pi_1(M)\longrightarrow\pi_1^{\mathrm{orb}}(\Sigma)$ induced by a holomorphic fibration $f:M\to\Sigma$.
\end{theorem}

We also need the following simple lemma on the cohomological dimension for discrete subgroups of Lie groups (See e.g., \cite[Chapter VIII, Section 9.4]{Brown1982CohomologyOG}).

\begin{lemma}\label{lem:cd-tf-cocompact-lattice-Lie}
Let $\Gamma$ be a torsion-free discrete subgroup of a real Lie group $G$ of non-compact type. Suppose $G$ has finitely many connected components. We have 
$$\mathrm{cd}(\Gamma)\leq \dim_{\mathbb R} X,$$
where $X$ is the symmetric space corresponding to $G$. Moreover, the equality holds if and only if $\Gamma$ is cocompact in $G$.
\end{lemma}

\begin{proof}
For the maximal compact subgroup $K$ of $G$, by the Second Manifold Splitting Theorem \cite[Theorem 14.3.11]{Hilgert-Neeb12}, we have $X\coloneqq G/K$ is diffeomorphic to some $\mathbb{R}^m$, where $m=\dim_{\mathbb R}G-\dim_{\mathbb R}K$. Therefore, $X/\Gamma$ is a Eilenberg–MacLane space $K(\Gamma,1)$. Then the lemma follows from \cite[Chapter VIII, Proposition 8.1-(a)]{Brown1982CohomologyOG}, which says that for any $d$-dimentional aspherical manifold $Y$, $\mathrm{cd}(\pi_1(Y))\le d$, with equality if and only if $Y$ is closed.
\end{proof}

\begin{theorem}
\label{prop:G2-not-kahler}
Let $G$ be a connected semisimple real Lie group with finite
center and no compact factors, with a factor isomorphic to $\operatorname{Spin}(n,1)$, $n\ge 4$ and is even. Then any torsion-free, irreducible and cocompact lattice $\Gamma_2<G$ is not a K\"ahler group.
\end{theorem}
We make the following observation before the proof of Proposition \ref{prop:G2-not-kahler}. Suppose $G = \operatorname{Spin}(n,1)\times H$, where the product $H$ of the other factor groups is possibly trivial. For even $n$, the natural projection map
\[
\operatorname{SO}(n,1)\longrightarrow 
\operatorname{PO}(n,1)\cong \operatorname{Isom}(\mathbb H^{n})
\]
is an isomorphism, since $-I_{n+1}\notin \operatorname{SO}(n,1)$.
Hence the standard double covering induces
\[
\theta_0:\operatorname{Spin}(n,1)\longrightarrow
\operatorname{SO}_0(n,1)
<\operatorname{Isom}(\mathbb H^{n}),
\]
with kernel $\{\pm1\}$. Let $\theta_H$ be the natural quotient map $H\to H/Z(H)\coloneqq H_{\mathrm{ad}}$, and consider the product homomorphism $\theta\coloneqq \theta_0\times \theta_H$.  Since $\Gamma_2$ is a torsion-free subgroup of $\operatorname{Spin}(n,1)\times H$, we have $\Gamma_2\cap\ker\theta=\{1\}$, i.e., 
\[
\theta|_{\Gamma_2}\colon \Gamma_2\to \SO_0(n,1)\times H_{\mathrm{ad}}
\]
is injective. Since $\Gamma_2$ is a cocompact subgroup of $\operatorname{Spin}(n,1)\times H$ and $\theta_0: \operatorname{Spin}(n,1)\to \SO_0(n,1)$ is a double cover, one can view $\Gamma_2\cong \theta(\Gamma_2)$ as a torsion-free cocompact subgroup of $\SO_0(n,1)\times H_{\mathrm{ad}}$. Since $\theta$ is a quotient by a finite normal subgroup, $\Gamma_2$ is also irreducible in $\SO_0(n,1)\times H_{\mathrm{ad}}$. We require the following lemma.
\begin{lemma}\label{lem:dense_and_inj}
    For the factor map $\vartheta: \SO_0(n,1)\times H_{\mathrm{ad}}\to \SO_0(n,1)$, $\vartheta(\Gamma_2)<\SO_0(n,1)$ is Zariski dense, and the restriction map $\widetilde{\vartheta}: \Gamma_2\to \SO_0(n,1)$ is injective.
\end{lemma}
\begin{proof}
For our first claim, if $H$ is trivial, then $\vartheta$ is the identity map, and the conclusion comes from the Borel density theorem \cite{Borel1960DensityPF}. In other cases, the irreducibility of $\Gamma_2$ implies the density of every $\vartheta(\Gamma_2)$ under the usual topology, and the Zariski density follows.

For our second claim, the conclusion is instant when $H$ is trivial. In other cases,
it suffices to show that $T\coloneqq \ker (\widetilde{\vartheta}) = \Gamma_2\cap (\{1\}\times H_{\mathrm{ad}})$ is trivial. In fact, consider the map 
\[
\widetilde{\vartheta}'\colon \Gamma_2\hookrightarrow \SO_0(n,1)\times H_{\mathrm{ad}}\to H_{\mathrm{ad}}
\]
where the second map is the natural projection onto the $H_{\mathrm{ad}}$ factor. The subgroup $\widetilde\vartheta'(T)$ is normal in
$\widetilde\vartheta'(\Gamma_2)$. Since $T$ is discrete and
$\widetilde\vartheta'|_T$ is injective, $\widetilde\vartheta'(T)$ is
discrete, hence closed in $H_{\mathrm{ad}}$. Since $\Gamma_2$ is irreducible, $\widetilde\vartheta'(\Gamma_2)$ is
dense in $H_{\mathrm{ad}}$, it follows that $\widetilde\vartheta'(T)$ is normal in
$H_{\mathrm{ad}}$. By \cite[Proposition 1.93(d)]{knapp1996lie}, $\widetilde{\vartheta}'(T)$ is contained in the center of $H_{\mathrm{ad}}$, which is trivial. We deduce that $T=\{1\}$ from these facts.
\end{proof}

\begin{proof}[Proof of Theorem \ref{prop:G2-not-kahler}]
Suppose by contradiction that $\Gamma_2$ is a K\"ahler group. Note $n$ is even, so there is a homomorphism induced from Lemma \ref{lem:dense_and_inj}:
\[
\rho\colon \Gamma_2\xrightarrow{\widetilde\vartheta} \SO_0(n,1)\hookrightarrow \SO(n,1)\cong\operatorname{Isom}(\mathbb H^{n}),
\]
so that the image of $\rho$ is Zariski dense since the Zariski closure $\overline{\SO_0(n,1)}=\SO(n,1)$. Implied by Theorem \ref{thm:DP}, $\rho$ factors through a map $\Gamma_2\to \pi_1^{\mathrm{orb}}(\Sigma)$. The injectivity of $\rho$ implies that $\Gamma_2<\pi_1^{\mathrm{orb}}(\Sigma)$, hence $\mathrm{cd}(\Gamma_2)\leq 2$.

On the other hand, as a cocompact torsion-free discrete lattice in $\SO_0(n,1)\times H_{\mathrm{ad}}$, Lemma \ref{lem:cd-tf-cocompact-lattice-Lie} implies
$\mathrm{cd}(\Gamma_2) = \dim_\mathbb{R}\bigl(\mathbb{H}^{n}\times (X/K_H)\bigr) \geq n\geq 4$,
a contradiction. Hence $\Gamma_2$ is not a K\"ahler group.
\end{proof}

\begin{proposition}\label{FiniPreseResidFini}
$\Gamma_1$ and $\Gamma_2$ are finitely presented and residually finite.
\end{proposition}
\begin{proof}
Since $\Gamma_1$ and $\Gamma_2$ are torsion-free cocompact lattices, they are the fundamental groups of the compact manifolds. Hence, both groups are finitely presented. Being finitely generated linear groups over a field of characteristic zero, they are residually finite by Mal'cev's theorem.
\end{proof}

\begin{theorem}\label{mainThem2}
    Let $n\ge 2$, $r_1\ge 0$, and $r_2\ge 1$ be integers such that $2r_1+r_2\ge 2$ and either $n$ is odd or $r_2$ is even. There exist finitely presented, residually finite, torsion-free cocompact arithmetic lattices 
    \[
    \Gamma_1< \operatorname{Spin}(2n-1,2)^{r_1+r_2}, \Gamma_2< \operatorname{Spin}(2n-1,2)^{r_1}\times \operatorname{Spin}(2n,1)^{r_2},
    \]
    with isomorphic profinite completions, $\widehat{\Gamma}_1\cong\widehat{\Gamma}_2$, such that $\Gamma_1$ is K\"ahler and $\Gamma_2$ is non-K\"ahler. In addition, if $r_1\ge 1$, $\Gamma_1$ and $\Gamma_2$ can be explicitly taken as certain level-$3$ congruence subgroups of the corresponding integral spin groups.
\end{theorem}
\begin{proof} 
We first show $\widehat{\Gamma}_1\cong\widehat{\Gamma}_2$ by applying Theorem \ref{thm:ss_classical}-(1). Here we have the signature pairs $(m_1,n_1) = (2n-1,2)$ and $(m_2,n_2) = (1,2n)$.
Note $d=2n+1\ge5$, and $2r_1+r_2\geq 2$ implies the required rank condition. Now if $n$ is odd, then $4\mid\bigl(2n-1)-1$; if $r_2$ is even, then
$2\mid\bigl(2n-1)-1,$ hence the congruence conditions hold. So it follows from Theorem \ref{thm:ss_classical}-(1), there exist torsion-free cocompact arithmetic lattices $\Gamma_1< \operatorname{Spin}(2n-1,2)^{r_1+r_2}$, $\Gamma_2< \operatorname{Spin}(2n-1,2)^{r_1}\times \operatorname{Spin}(2n,1)^{r_2}$ with $\widehat{\Gamma}_1\cong\widehat{\Gamma}_2$. If $r_1\ge 1$, Theorem \ref{prop:lattice_level_3} ensure that $\Gamma_1$, $\Gamma_2$ can be taken as certain level-$3$ congruence subgroups.

Now, directly from Proposition \ref{prop:G1-kahler} and Theorem \ref{prop:G2-not-kahler}, we can see that $\Gamma_1$ is K\"ahler and $\Gamma_2$ is non-K\"ahler. Moreover, by Proposition \ref{FiniPreseResidFini}, $\Gamma_1,\Gamma_2$ are finitely presented and residually finite.
\end{proof}
We conclude the paper with the following concrete example of lattices,
\[
\Gamma_1< \operatorname{Spin}(3,2)^3\cong \operatorname{Sp}(4,\mathbb{R})^3,\quad \Gamma_2< \operatorname{Spin}(3,2)\times \operatorname{Spin}(4,1)^2\cong \operatorname{Sp}(4,\mathbb{R})\times \operatorname{Sp}(1,1)^2.
\]
\begin{example}
    Let $\alpha$ be a root of $f(x)=x^4-4x^2+x+1$, and put $k=\mathbb Q(\alpha)$. The properties of $f$ imply that both $\alpha$ and $\alpha-1$ are algebraic units and that $k_v/\mathbb Q_2$ is unramified at the unique place $v\mid2$. Moreover, $\alpha$ lies in $(1,\infty)$ at one real place, in $(0,1)$ at one real place, and is negative at the other two. Define
    \[
    q_1=\langle1^{[3]},(\alpha-1)^{[2]}\rangle,
    \qquad
    q_2=\langle1,\alpha^{[2]},(\alpha-1)^{[2]}\rangle,
    \]
    and let $\mathbf G_i=\operatorname{Spin}(k,q_i)$ for $i=1,2$. The products of their noncompact archimedean factors are
    \[
    G_1\cong\operatorname{Spin}(3,2)^3,
    \qquad
    G_2\cong\operatorname{Spin}(3,2)
    \times\operatorname{Spin}(4,1)^2.
    \]
    These forms satisfy the hypotheses of Lemma~\ref{lem:ss_local} and Proposition~\ref{prop:ss_global} for the second alternative of Theorem~\ref{thm:ss_classical}(1).

    Let $\Gamma_i\leq\mathbf G_i(k)$ be the level-$3$ congruence subgroups defined in Subsection~\ref{subsec:explicit}. Moreover, $\operatorname{disc}(f)=1957$ is square-free, so $\mathcal O_k=\mathbb{Z}[\alpha]$ in this definition. Identifying these congruence subgroups with their images under the archimedean projections, we obtain arithmetic lattices
    \[
    \Gamma_1<\operatorname{Spin}(3,2)^3,
    \qquad
    \Gamma_2<\operatorname{Spin}(3,2)
    \times\operatorname{Spin}(4,1)^2.
    \]
    Since both $G_1$ and $G_2$ have a factor $\operatorname{Spin}(3,2)$, Proposition~\ref{prop:lattice_level_3} implies that $\Gamma_1$ and $\Gamma_2$ have isomorphic profinite completions. All factors of $G_1$ are $\operatorname{Spin}(3,2)$, whereas $G_2$ has factors $\operatorname{Spin}(4,1)$. Thus Proposition ~\ref{prop:G1-kahler} and Theorem~\ref{prop:G2-not-kahler} imply that $\Gamma_1$ is K\"ahler, whereas $\Gamma_2$ is not.
\end{example}

\noindent \textbf{AI Declaration} The mathematical ideas and proofs of this paper are from the authors. AI was only used for typesetting and language clean-up. We take the responsibility for the content.

\appendix
\section{Overview of the Noncompact Real Forms of Lie Groups of Dynkin types \texorpdfstring{$A_l$}{Al}, \texorpdfstring{$B_l$}{Bl}, \texorpdfstring{$C_l$}{Cl}, and \texorpdfstring{$D_l$}{Dl}}
\label{appendix:a}
To demonstrate that Proposition \ref{prop:classic}, Proposition \ref{prop:slh} and Theorem \ref{thm:rigidity} fully characterizes the profinite rigidity for lattices in real forms of Lie groups of types $A_l$, $B_l$, $C_l$, and $D_l$ of rank $\geq 2$, we summarize the groups in the following table, and provide the relevant information of them. The classification and the symmetric space dimensions are standard; for the inner and outer forms of each Lie group, see the description of the classical absolutely simple groups in \cite{tits1966classification}. Note in the Table below, $\ell$ denotes the absolute rank, that is, the index in the absolute Dynkin type; it should not be confused with the real rank. For a family parametrized by a signature $(m,n)$, we write $d=m+n$.

\begin{table}[H]
\centering
\begin{tabular}{c | c | c | c | c}
\toprule
Type & Cartan type & Noncompact real form & 
$\dim(G/K)$ & Inner to compact form? \\
\midrule

\multirow{3}{*}{$A_l$}
& AI
& $\mathrm{SL}({l+1},\mathbb{R})$
& $\dfrac{l(l+3)}{2}$
& No \\
\cmidrule(lr){2-5}

& AII
& $\mathrm{SL}\left(\dfrac{l+1}{2},\mathbb{H}\right)$
& $\dfrac{(l-1)(l+2)}{2}$
& No \\
\cmidrule(lr){2-5}

& AIII
& $\mathrm{SU}(m,n)$,\quad $m+n=l+1$
& $2mn$
& Yes \\

\midrule

$B_l$
& BI
& $\mathrm{Spin}(m,n)$,\quad $m+n=2l+1$
& $mn$
& Yes \\

\midrule

\multirow{2}{*}{$C_l$}
& CI
& $\mathrm{Sp}({2l},\mathbb{R})$
& $l(l+1)$
& Yes \\
\cmidrule(lr){2-5}

& CII
& $\mathrm{Sp}(m,n)$,\quad $m+n=l$
& $4mn$
& Yes \\

\midrule

\multirow{2}{*}{$D_l$}
& DI
& $\mathrm{Spin}(m,n)$,\quad $m+n=2l$
& $mn$
& $\begin{array}{c}
\text{Yes if }m,n\text{ are even},\\
\text{No if }m,n\text{ are odd}
\end{array}$ \\
\cmidrule(lr){2-5}

& DIII
& $\mathrm{Spin}^{*}(2l)$
& $l(l-1)$
& Yes \\

\bottomrule
\end{tabular}
\caption{Noncompact real forms of the classical simple Lie groups and the
dimensions of their associated symmetric spaces.}
\label{tab:classical-real-forms}
\end{table}

\bibliographystyle{alpha}

\bibliography{sample}

@article{wienhard2005bounded,
  author = {Wienhard, A.},
  title = {Bounded cohomology and geometry},
  journal = {arXiv preprint math/0501258},
  year = {2005}
}

@article{guichardet1978cohomologie,
  author = {Guichardet, A. and Wigner, D.},
  title = {Sur la cohomologie r{\'e}elle des groupes de Lie simples r{\'e}els},
  journal = {Ann. Sci. {\'E}c. Norm. Sup{\'e}r.},
  volume = {11},
  number = {2},
  pages = {277--292},
  year = {1978}
}

@article{kammeyer2023profinite,
  author = {Kammeyer, H. and Kionke, S.},
  title = {On the profinite rigidity of lattices in higher rank Lie groups},
  journal = {Math. Proc. Cambridge Philos. Soc.},
  volume = {174},
  number = {2},
  pages = {369--384},
  year = {2023}
}

@book{lam2005introduction,
  author = {Lam, T.-Y.},
  title = {Introduction to quadratic forms over fields},
  volume = {67},
  publisher = {American Mathematical Society},
  year = {2005}
}

@book{knapp1996lie,
  author = {Knapp, A. W.},
  title = {Lie groups beyond an introduction},
  series = {Progress in Mathematics},
  publisher = {Birkh{\"a}user},
  year = {1996}
}

@book{Hilgert-Neeb12,
  author = {Hilgert, J. and Neeb, K.-H.},
  title = {Structure and geometry of Lie groups},
  series = {Springer Monographs in Mathematics},
  publisher = {Springer},
  year = {2012}
}

@book{Brown1982CohomologyOG,
  author = {Brown, K. S.},
  title = {Cohomology of groups},
  series = {Graduate Texts in Mathematics},
  volume = {87},
  publisher = {Springer},
  year = {1982}
}

@article{borel1962arithmetic,
  author = {Borel, A. and Harish-Chandra},
  title = {Arithmetic subgroups of algebraic groups},
  journal = {Ann. of Math. (2)},
  volume = {75},
  number = {3},
  pages = {485--535},
  year = {1962}
}

@article{Gromov1989SurLG,
  author = {Gromov, M.},
  title = {Sur le groupe fondamental d'une vari{\'e}t{\'e} k{\"a}hl{\'e}rienne},
  journal = {C. R. Acad. Sci. Paris S{\'e}r. I Math.},
  volume = {308},
  pages = {67--70},
  year = {1989}
}

@article{Borel1960DensityPF,
  author = {Borel, A.},
  title = {Density properties for certain subgroups of semi-simple groups without compact components},
  journal = {Ann. of Math. (2)},
  volume = {72},
  pages = {179--188},
  year = {1960}
}

@article{hughes2025profinite,
  author = {Hughes, S. and Llosa Isenrich, C. and Py, P. and Stover, M. and Vidussi, S.},
  title = {Profinite rigidity of K{\"a}hler groups: Riemann surfaces and subdirect products},
  journal = {arXiv preprint arXiv:2501.13761},
  year = {2025}
}

@article{grothendieck5geometrie,
  author = {Grothendieck, A.},
  title = {G{\'e}om{\'e}trie formelle et g{\'e}om{\'e}trie alg{\'e}brique},
  journal = {S{\'e}minaire Bourbaki},
  volume = {5},
  pages = {193--220},
  year = {1960}
}

@article{grothendieck1970representations,
  author = {Grothendieck, A.},
  title = {Repr{\'e}sentations lin{\'e}aires et compactification profinie des groupes discrets},
  journal = {Manuscripta Math.},
  volume = {2},
  number = {4},
  pages = {375--396},
  year = {1970}
}

@article{bridson2004grothendieck,
  author = {Bridson, M. R. and Grunewald, F. J.},
  title = {Grothendieck's problems concerning profinite completions and representations of groups},
  journal = {Ann. of Math. (2)},
  volume = {160},
  number = {1},
  pages = {359--373},
  year = {2004}
}

@article{serre1964exemples,
  author = {Serre, J.-P.},
  title = {Exemples de vari{\'e}t{\'e}s projectives conjugu{\'e}es non hom{\'e}omorphes},
  journal = {C. R. Acad. Sci. Paris},
  volume = {258},
  pages = {4194--4196},
  year = {1964}
}

@article{GDZPPN002196646,
  author = {Kneser, M.},
  title = {Normalteiler ganzzahliger Spingruppen},
  journal = {J. Reine Angew. Math.},
  volume = {311/312},
  pages = {191--214},
  year = {1979},
  zbl = {0409.20038}
}

@article{delzant2012kahler,
  author = {Delzant, T. and Py, P.},
  title = {K{\"a}hler groups, real hyperbolic spaces and the Cremona group. With an appendix by {S}erge {C}antat},
  journal = {Compos. Math.},
  volume = {148},
  number = {1},
  pages = {153--184},
  year = {2012}
}

@article{echtler2024bounded,
  author = {Echtler, D. and Kammeyer, H.},
  title = {Bounded cohomology is not a profinite invariant},
  journal = {Canad. Math. Bull.},
  volume = {67},
  number = {2},
  pages = {379--390},
  year = {2024}
}

@inproceedings{tits1966classification,
  author = {Tits, J.},
  title = {Classification of algebraic semisimple groups},
  booktitle = {Proc. Sympos. Pure Math.},
  pages = {33--62},
  year = {1966}
}

@article{kammeyer2020profinite,
  author = {Kammeyer, H. and Kionke, S. and Raimbault, J. and Sauer, R.},
  title = {Profinite invariants of arithmetic groups},
  journal = {Forum Math. Sigma},
  volume = {8},
  pages = {e54},
  year = {2020}
}

@article{kammeyer2021adelic,
  author = {Kammeyer, H. and Kionke, S.},
  title = {Adelic superrigidity and profinitely solitary lattices},
  journal = {Pacific J. Math.},
  volume = {313},
  number = {1},
  pages = {137--158},
  year = {2021}
}

@article{schwermer2011geometric,
  author = {Schwermer, J.},
  title = {Geometric cycles, Albert algebras and related cohomology classes for arithmetic groups},
  journal = {Groups Geom. Dyn.},
  volume = {5},
  number = {2},
  pages = {529--552},
  year = {2011}
}

@article{garibaldi2007groups,
  author = {Garibaldi, S. and Petersson, H. P.},
  title = {Groups of outer type {$E_6$} with trivial Tits algebras},
  journal = {Transform. Groups},
  volume = {12},
  number = {3},
  pages = {443--474},
  year = {2007}
}

@book{jacobson2017exceptional,
  author = {Jacobson, N.},
  title = {Exceptional Lie algebras},
  series = {Lecture Notes in Pure and Applied Mathematics},
  volume = {1},
  publisher = {Marcel Dekker},
  address = {New York},
  year = {1971}
}

@article{tits1965algebres,
  author = {Tits, J.},
  title = {Alg{\`e}bres alternatives, alg{\`e}bres de Jordan et alg{\`e}bres de Lie exceptionnelles. I. Construction},
  journal = {Indag. Math.},
  volume = {28},
  pages = {223--237},
  year = {1966}
}

@article{barton2003magic,
  author = {Barton, C. H. and Sudbery, A.},
  title = {Magic squares and matrix models of Lie algebras},
  journal = {Adv. Math.},
  volume = {180},
  number = {2},
  pages = {596--647},
  year = {2003}
}

@book{jacobson2013lie,
  author = {Jacobson, N.},
  title = {Lie algebras},
  publisher = {Dover Publications},
  address = {New York},
  pages = {ix+331},
  year = {1979},
  note = {Republication of the 1962 original}
}

@incollection{Rap99,
  author    = {Rapinchuk, Andrei S.},
  title     = {The congruence subgroup problem},
  booktitle = {Algebra, K-Theory, Groups, and Education},
  series    = {Contemporary Mathematics},
  volume    = {243},
  pages     = {175--188},
  publisher = {American Mathematical Society},
  address   = {Providence, RI},
  year      = {1999},
  doi       = {10.1090/conm/243/03693}
}

@book{lubotzky2003subgroup,
  title={Subgroup growth},
  author={Lubotzky, Alexander and Segal, Dan},
  volume={212},
  year={2003},
  publisher={Springer}
}

@article{aka2012profinite,
  title={Profinite completions and Kazhdan’s property (T)},
  author={Aka, Menny},
  journal={Groups, Geometry, and Dynamics},
  volume={6},
  number={2},
  pages={221--229},
  year={2012}
}

@article{baumann2026congruence,
  title={Congruence rigidity of algebraic groups},
  author={Baumann, Adrian and Kammeyer, Holger},
  journal={arXiv preprint arXiv:2605.03107},
  year={2026}
}

@article{prasad2010developments,
  title={Developments on the congruence subgroup problem after the work of Bass, Milnor and Serre},
  author={Prasad, Gopal and Rapinchuk, Andrei S},
  journal={American Mathematical Society, Providence, RI},
  year={2010}
}

@book{platonov1994algebraic,
  title={Algebraic groups and number theory},
  author={Platonov, Vladimir and Rapinchuk, Andre{\u\i}}}

@article{he2023k,
  title={On k-universal quadratic lattices over unramified dyadic local fields},
  author={He, Zilong and Hu, Yong},
  journal={Journal of Pure and Applied Algebra},
  volume={227},
  number={7},
  pages={107334},
  year={2023},
  publisher={Elsevier}
}

\end{document}